\documentclass[a4paper,reqno]{amsart}

\usepackage[active]{srcltx}
\usepackage[all]{xy}
\usepackage{subfig}
\usepackage{hyperref}

\usepackage{amssymb,latexsym}
\usepackage{mathrsfs, esint}
\usepackage{graphics}
\usepackage{latexsym}
\usepackage{psfrag}
\usepackage{verbatim}
\usepackage[dvips]{graphicx}
\usepackage{color}
\usepackage{pifont,marvosym}
\usepackage{psfrag}

\theoremstyle{plain}
\newtheorem{lemma}{Lemma}[section]
\newtheorem{proposition}[lemma]{Proposition}
\newtheorem{theorem}[lemma]{Theorem}
\newtheorem{corollary}[lemma]{Corollary}

\theoremstyle{definition}
\newtheorem{definition}[lemma]{Definition}

\theoremstyle{remark}
\newtheorem{remark}[lemma]{Remark}

\numberwithin{equation}{section}

\def\N{\mathbb{N}}

\def\Z{\mathbb{Z}}
\def\R{\mathbb{R}}
\def\C{\mathbb{C}}
\def\CC{\mathcal{C}}

\def\md{\mathring \delta}
\def\mR{\mathring R}
\def\T{\mathbb{T}^3}

\def\SS#1{\mathcal{S}_{#1}^{3\times3}}
\def\S{\mathbb{S}}

\def\Id{\mathrm{Id}}

\def\to{\rightarrow}
\def\tr{\mathrm{tr}}
\def\div{\mathrm{div}\,}
\def\inter{\mathrm{int}}
\def\eps{\varepsilon}
\def\supp{\mathrm{supp}\,}

\def\XXint#1#2#3{{\setbox0=\hbox{$#1{#2#3}{\int}$} \vcenter{\vspace{-1pt}\hbox{$#2#3$}}\kern-.5\wd0}}

\def\ev#1{\Big(e-\int|v_{#1}|^2\Big)}

\begin{document}
 
\title[Cauchy problem for H\"older solutions to the Euler equations]{Cauchy problem for dissipative H\"older solutions to the incompressible Euler equations} 

\author{Sara Daneri} \address{Institut f\"ur Mathematik, Universit\"at Z\"urich, CH-8057 Z\"urich}\email{sara.daneri@math.uzh.ch}

\begin{abstract}
We consider solutions to the Cauchy problem for the incompressible Euler equations on the $3$-dimensional torus which are continuous or H\"older continuous for any exponent $\theta<\frac{1}{16}$. Using the techniques introduced in \cite{DS12} and \cite{DS12H}, we prove the existence of infinitely many (H\"older) continuous initial vector fields starting from which there exist infinitely many (H\"older) continuous solutions with preassigned total kinetic energy.
\end{abstract}

\maketitle
\tableofcontents

\section{Introduction}
\label{S:intro}

In this paper we deal with continuous and H\"older solutions of the Cauchy problem for the Euler equations on the 3-dimensional torus $\T=\S^1\times\S^1\times\S^1$

\begin{equation}
\label{E:euler}
 \left\{\begin{aligned}
         &\partial_tv+\div(v\otimes v)+\nabla p=0 && &\text{in $\T\times(0,1)$}\\
&\div v=0 && &\text{in $\T\times(0,1)$}\\
&v(\cdot,0)=v_0 && &\text{in $\T$}.
        \end{aligned}
\right.
\end{equation}
A pair $(v,p)\in C^0(\T\times[0,1];\R^3\times\R)$ is a \emph{continuous solution} of \eqref{E:euler} with initial datum $v_0\in C(\T;\R^3)$ if it satisfies \eqref{E:euler} in the weak distributional sense. Equivalently, on all simply connected subdomains $U\subset \T$ with $C^1$ boundary and for all $t\in (0,1)$

\begin{align*}
 &\int_Uv_0(x)\,dx=\int_Uv(x,t)\,dx+\int_0^t\int_{\partial U}[v(v\cdot\nu)+p\nu](x,s)\,dS(x)ds,\\
&\int_{\partial U}[v\cdot\nu](x,t)\,dS(x)=0,
\end{align*}
being $\nu$ the outer unit normal to $\partial U$ and $dS$ the surface measure on $\partial U$.

Moreover, a continuous solution to \eqref{E:euler} is an \emph{H\"older solution} with exponent $\theta\in(0,1)$ if $\exists\,C>0$ s.t. 
\begin{equation}
\label{E_cthetasol}
|v(x,t)-v(x',t)|\leq C|x-x'|^\theta,\quad\forall\,x,x'\in\T,\,\forall\,t\in[0,1].
\end{equation}

Given a nonnegative continuous function $e:[0,1]\to\R$, we say that $v:\T\times[0,1]\times\R^3$ has \emph{total kinetic energy} $e$ if 
\[
 \int_{\T}|v(x,t)|^2\,dx=e(t),\quad\forall\,t\in[0,1].
\]

It is known since the '20s (see \cite{Gun} \cite{Lic}) that, if $v_0$ is sufficiently smooth, then there exists a unique classical solution of \eqref{E:euler} on a time interval $[0,T]$, $T=T\bigl({\sup}\{|v_0(x)|:\,x\in\T\}\bigr)>0$ (for a modern result assuming $v_0\in H^s(\T)$ with $s>2$, see e.g. \cite{BM02}). Moreover, its total kinetic energy is constant.

However, in 1949 Onsager first conjectured the existence of weak solutions which are dissipative, namely whose total kinetic energy is monotone decreasing. This phenomenon, that in fluid dynamics literature is called ``anomalous dissipation'', is consistent with the energy inequalities satisfied by weak limits of Leray solutions of the Navier-Stokes equations.
In \cite{Ons} Onsager stated more precisely that $C^{0,\theta}$-solutions are conservative when $\theta>\frac{1}{3}$, while there exist dissipative $C^{0,\theta}$ solutions for any $\theta<\frac{1}{3}$. By $C^{0,\theta}$ solutions we mean vector fields $v\in C^0(\T\times[0,1];\R^3)$ satisfying \eqref{E_cthetasol}.

The first part of the conjecture was completely settled by Eynik \cite{Eyi} and by Constantine, E and Titi \cite{CET}. 

The second part is, in its full generality, still open. The first example of weak solution violating the energy conservation was given by Scheffer \cite{Sch}, who showed the existence of a nontrivial compactly supported weak solution in $\R^2\times \R$. A different example of nontrivial compactly supported weak solution in $\mathbb T^2\times\R$ was then given by Shnirelman in \cite{Shn97}. In both cases the solutions are only square summable and it is not clear whether there are time intervals in which their total kinetic energy is monotone decreasing. The first proof of existence of solutions with monotone decreasing total kinetic energy was given by Shnirelman in \cite{Shn00}. This solution belongs to the energy space $L^\infty([0,+\infty);L^2(\R^3))$. 

In \cite{DSAnn07}, De Lellis and Sz\'ekelyhidi proved the existence of nontrivial compactly supported bounded weak solutions in any space dimension. Moreover, such solutions can attain a prescribed total kinetic energy for almost every time $t\in[0,+\infty)$. The full control of the total kinetic energy for all times was finally achieved in \cite{DSArma}, in which they could prove the existence of compactly supported solutions $(v,p)\in L^\infty(\R^d\times [0,+\infty);\R^d\times\R)$ with $v\in C([0,+\infty);L^2(\R^d;\R^d))$ and
\begin{equation*}
 \int_{\R^d}|v|^2(t)=e(t),\quad\forall\,t\in[0,+\infty),
\end{equation*}
being $e:[0,T]\to\R$ any positive continuous function given a priori. 
In particular, setting $v_0=\underset{t\to0}{\lim}\,v(t)$ in the $L^2$ norm, $(v,p)$ is a bounded weak solution of the Cauchy problem \eqref{E:euler} with prescribed total kinetic energy $e$. As a corollary, choosing $e$ to be a monotone decreasing function, they obtain the existence of bounded dissipative solutions in all dimensions, thus extending Shnirelman's result \cite{Shn00}. Given a suitable initial datum $v_0$ for which such solution exists, De Lellis and Sz\'ekelyhidi's method provides for an infinite set of solutions with the same total kinetic energy. Hence, on the one hand, it turns out that none of the admissibility criteria based on energy inequalities which have been proposed for the Euler equations is able to single out a unique solution for arbitrary $L^\infty$ initial data (see also \cite{BarTiti10} for explicit examples of $L^2$ weak solutions of \eqref{E:euler} having constant total kinetic energy). On the other hand, by the local existence and uniqueness of classical solutions and the weak-strong uniqueness of the so-called \emph{admissible measure-valued solutions} of \eqref{E:euler} proved in \cite{BreDeLSze} --including all dissipative $L^2$ weak solutions-- the vector fields $v_0$ for which one has such severe loss of uniqueness can not be too regular. These initial data --which in \cite{DSArma} are called \emph{wild initial data}-- are nonetheless proved to be a dense set in the space of $L^2$ divergence free vector fields \cite{SzW12} and they include also the classical vortex sheet (see \cite{Sz11}).
%(for the existence of global weak solutions with bounded energy though discontinuous at $t=0$ see also \cite{Wie}). 

In \cite{DS12} and \cite{DS12H} De Lellis and Sz\'ekelyhidi developed their method up to prove the existence of respectively continuous and $C^{0,\theta}$, with $\theta<\frac{1}{10}$, solutions having prescribed total kinetic energy on the $3$-dimensional torus. 
In \cite{ChDLS} the results of \cite{DS12} are extended to the $2$-dimensional case.
Recently, Isett \cite{Ise12} did one step forward towards the proof of Onsager's conjecture, constructing global weak solutions to the $3$-dimensional incompressible Euler equations which are zero outside of a finite time interval and have velocity in the H\"older class $C^{0,\theta}$ for every $\theta<\frac{1}{5}$. In \cite{BuDLS} the authors have considerably simplified the proof of \cite{BuDLS}, though taking advantage of its main new ideas, and showed the existence of $C^{0,\theta}$ solutions whose total kinetic energy is dissipated for any $\theta<\frac{1}{5}$. 
 
The aim of this paper is to extend the results of \cite{DSArma} showing that the set of initial data for which one can get infinitely many H\"older solutions which dissipate the total kinetic energy, is an infinite subset of $C^{0,\theta}(\T;\R^3)$, for $\theta$ smaller than a suitable constant. Using the estimates provided in \cite{DS12H}, we get such result for any exponent $\theta<\frac{1}{16}$, but we believe that the ideas of our method, when implemented with finer estimates (see e.g. \cite{Ise12} and \cite{BuDLS}), can raise the threshold of non-uniqueness for the Cauchy problem up to the same exponent as for the incompressible Euler equations with no preassigned initial data. 

\begin{theorem}
 \label{T:main}
 Let $e:[0,1]\to\R$ be a positive smooth function. Then, for any $\theta<\frac{1}{16}$ there exist infinitely many $v_0\in C^{0,\theta}(\T;\R^3)$ satisfying $e(0)=\int_{\T}|v_0|^2$ and each being the initial datum of infinitely many $(v,p)\in C^0(\T\times[0,1];\R^3\times\R)$ solving \eqref{E:euler} and satisfying
\begin{align*}
 &|v(x,t)-v(x',t)|\leq C|x-x'|^\theta,\quad\forall\,x,x'\in\T,\,t\in[0,T]\\
&\int_{\T}|v(x,t)|^2\,dx=e(t),\quad\forall\,t\in[0,T].
\end{align*}
\end{theorem}

In particular, for the same reasons as in the case of dissipative $L^\infty$ weak solutions, whenever $e$ in nonincreasing then the initial data $v_0$ for which Theorem \ref{T:main} holds cannot be too regular.

The underlying ideas of the method used in the proof of Theorem \ref{T:main} are the ones introduced in \cite{DSAnn07} and used as well in \cite{DSArma}, \cite{DS12} and \cite{DS12H}.
Besides considerably improving the previous results, the approach introduced by De Lellis and Sz\'ekelyhidi was new in this field and revealed unexpected connections between non-uniqueness phenomena appearing in some geometric problems --named by Gromov instances of the \emph{h-principle} \cite{Gro86}-- and non-uniqueness in PDE problems. In particular, the Onsager's conjecture has striking similarities with the rigidity and flexibility properties of isometric embeddings of Riemannian manifolds, first explored in the celebrated works of Nash \cite{Nash} and Kuiper \cite{Kui55} (on this topic see \cite{CDS} and the survey \cite{DLS11}). 

As in Nash and Kuiper's papers, the solutions of the Euler equations are generated by an iteration scheme. 
The iteration starts with a ``subsolution'' to the problem, that in \cite{DS12} and \cite{DS12H} is given by a solution of a perturbation of the Euler system (called Euler-Reynolds system) satisfying a suitable strict energy inequality w.r.t. a prescribed total kinetic energy $e$. The definition of subsolution is s.t. a solution of the Euler-Reynolds system whose perturbation term --which is called \emph{Reynolds stress tensor}-- is identically zero and for which the energy inequality holds as an equality is a solution of the Euler equations with total kinetic energy $e$. Then, at each step of the iteration one finds a new subsolution, by adding to the given velocity field a suitable fast oscillating perturbation plus a small corrector term that makes the new velocity field satisfy the Euler-Reynolds system. Moreover, the type of perturbation is s.t., if the oscillation parameters are large enough, then both the new velocity field and the pressure are arbitrarily close in $C^0$ to the previous ones and both the $C^0$ norm of the new Reynolds stress and the gap in the new energy inequality can be made arbitrarily small. If the oscillation parameters are chosen big enough, then the sequence of subsolutions converge in $C^0$ to a solution of the Euler system with total kinetic energy $e$. In \cite{DS12H} the authors are also able to control in the iteration process the growth of the $C^1$ norms in such a way to get, by interpolation, that the limiting vector field belongs to $C^{0,\theta}$ for a fixed $\theta<\frac{1}{10}$. The building blocks of the main perturbation term are Beltrami flows, a special class of stationary oscillatory solutions to the Euler equations, and the initial subsolution used to start the iteration is the trivial one. 

In our paper we use the same type of iteration scheme and estimates but modifying the notion of subsolution in order to include the information on the initial datum $\int |v(0)|^2=e(0)$ (see Definition \ref{D_admissible_subs}) and modifying as well the perturbations so as to leave the initial datum unchanged during the iteration procedure. This is done multiplying the perturbed Beltrami flows and their correction terms by suitable time dependent cut-off functions (see Section \ref{S_construction_step}). With the same ``time-localization'' trick, i.e. using the same type of iterative perturbation scheme but different cut-off functions, we also prove that there exist infinitely many nontrivial subsolutions (see Proposition \ref{P_subsolution}), hence infinitely many ``non-uniqueness initial data'' --that we call \emph{admissible initial data}-- as in Theorem \ref{T:main}. The fact that the solutions and subsolutions so obtained are infinitely many is a consequence of this general iterative perturbation scheme. 

In Section \ref{S_cont} we deal with continuous solutions of \eqref{E:euler}, namely no H\"older regularity is required. In this particular case we are able to provide, for the proof of Theorem \ref{T:main}, a notion of subsolution which is much less rigid than in the H\"older case. Since during the iteration procedure it is not necessary to keep under precise control the rate of growth of the $C^1$ norms, the strict energy inequality that we require in this case is less restrictive than in \cite{DS12} and \cite{DS12H}. Our notion of subsolution in this case is analogous, apart from the information on the initial data, to the one introduced in \cite{Cho}. 

For other non-uniqueness results in other partial differential equations obtained along the ideas first introduced in \cite{DSAnn07}
see e.g. \cite{Chi}, \cite{CFG}, \cite{Shv}, \cite{Sz12} and the survey \cite{DLS11}. Another application of these methods is to the existence of global weak solutions of the Euler equations with bounded energy, though discontinuous at $t=0$, proved in \cite{Wie}.
\subsection{Structure of the paper}
In Section \ref{S:mainpertstep} we introduce the Euler-Reynolds system \eqref{E:ersystem} and state Proposition \ref{P_pert_ctheta}, which is the building block of the generic iteration step in the proof of Theorem \ref{T:main}. Given a solution of the Euler-Reynolds system satisfying suitable conditions on the gap between its total kinetic energy and a given positive function $e$ and moreover on the $C^0$ norm of its Reynolds stress, one can perturb it on an arbitrary time subinterval thus getting another solution of the Euler-Reynolds system with a smaller energy gap and $C^0$ norm of the Reynolds stress tensor. The localization in time is represented by the multiplication of the perturbation by a smooth function $\psi$ with values in $[0,1]$. Moreover, the velocity and pressure fields of the new solution are arbirtrarily close in $C^0$ to the previous ones, with closeness parameters of the same order of the ones for the energy gap. In the mean time, also the growth of the $C^1$ norms of both the velocity and the Reynolds stress tensor can be controlled by suitable powers of the $C^1$-norms of the previous ones, of the $C^0$ closeness parameters and of the $C^2$ norm of $\psi$. In Corollary \ref{C_pert_ctheta}, we show that imposing suitable bounds for the derivatives of $\psi$ we can make all the estimates just dependent on the support of $\psi$ and not on its values, thus justifying the idea of a \emph{localized} perturbation.

In Section \ref{S:prelimdef} we recall some definitions and analytic estimates from \cite{DS12} and \cite{DS12H} which are preliminary to the proof of Proposition \ref{P_pert_ctheta}.

In Section \ref{S_construction_step} we define the perturbed solutions of the Euler-Reynolds system as in Proposition \ref{P_pert_ctheta}, and in Section \ref{S_doubling} we prove the related $C^0$, $C^1$ and energy estimates.

In Section \ref{S_subsolution} we define and prove the existence of infinitely many \emph{admissible initial data}, namely of H\"older continuous vector fields starting from which there exist infinitely many H\"older \emph{admissible subsolutions} (see Definition \ref{D_admissible_subs}.

Section \ref{S_solution} contains the proof of the main Theorem \ref{T:main}. More precisely, we prove that for any admissible subsolution w.r.t. a given total kinetic energy $e$ there exist infinitely many solutions of \eqref{E:euler} with the same initial datum and total kinetic energy $e$.

Section \ref{S_cont} gives another proof of Theorem \ref{T:main} for continuous solutions of \ref{E:euler}. In this case, though the basic ideas of the construction are the same, the estimates are simpler since we do not need to control in a quantitative way the growth of the $C^1$ norms. Moreover, it is worth noticing that the concept of admissible subsolution (see Definition \ref{D:contsubs}) is much more flexibile than in the H\"older case (see Definition \ref{D_admissible_subs}), since for example it does not require such a rigid relation between the elements of the  sequence of $C^0$-closeness parameters $\{\delta_n\}$.

\section*{Acknowledgments}
The author warmly thanks Camillo De Lellis for having proposed the problem and for fruitful discussions. 

\subsection{Notation}

If $f:\T\times[a_1,a_2]\to\R^d$ is continuous, then for all $t\in[a_1,a_2]$ define $f(t):\T\to\R^d$ as $f(t)(x)=f(x,t)$ and $\|f(t)\|_0=\underset{x\in\T}{\sup}|f(x,t)|$.
When clear from the context we will omit writing the time variable and set $f$, $\|f\|_0$ instead of $f(t)$, $\|f(t)\|_0$. Thus, the symbol $\|f\|$ will denote also the time-depending map giving the supremum norms w.r.t. the spatial variables $[a_1,a_2]\ni t\mapsto \|f(t)\|_0\in\R$.

Analogously, for every $r\in\N$ we denote by $\|f\|_r$ the time-dependent $C^r$ norm of $f$, and for any $r>0$ which is not an integer its H\"older $C^{[r],r-[r]}$ norm. The spatial H\"older $r$-seminorms of $f$ will be denoted by $[f]_r$. 
We also set, for $r\in\N$
\begin{equation*}
 |||f|||_r:=\underset{y\in\T}{\sup}|D^r_{t,x}f(y,t)|, \quad\forall\,t\in[a_1,a_2],
\end{equation*}
where $D^r_{t,x}f$ are the mixed derivates of order $r$ in $t,x$ and
\begin{align*}
 \|f\|_{C^r}&:=\underset{t\in[a_1,a_2]}{\sup}|||f|||_r,\\
 \|f\|_{C^0}&:=\underset{t\in[a_1,a_2]}{\sup}\|f\|_0,\\
\|f\|_{C^\alpha}&:=\underset{t\in[a_1,a_2]}{\sup}\|f\|_\alpha.
\end{align*}
In the same way, with $\int_{\T}f$ or simply $\int f$ we will denote the time-dependent map $[a_1,a_2]\ni t\mapsto \int_{\T}f(x,t)\,dx$, being $dx$ the Lebesgue measure on $\T$ with $\int_{\T}\,dx=(2\pi)^3$.

For $A\subset \S^2$, $\inter \,A$ denotes its interior in the relative topology induced by $\R^3$ on $\S^2$.
A set $A\subset\R^d$ is symmetric if $A=-A$.

We let $\SS{}$ be the space of symmetric matrices acting on $\R^3$, $\SS{+}$ the subset of those
 wich are positive definite and $\SS{0}$ the symmetric matrices with trace $0$. We also define 
\[
\mathbb M_3:=\Big\{\Id-b\otimes b:\,b\in\S^2\Big\}
\]
and the open set 
\[
 \mathcal M_3:=\inter\Big\{\underset{i=1}{\overset{m}{\sum}}a_iM_i:\,a_i>0,\,M_i\in\mathbb M_3,\,m\in\N\Big\}.
\]
Recall the following characterization of $\mathcal M_3$ from \cite{Cho}
\begin{equation}
\label{E:m3char}
 R\in\mathcal M_3\quad\Leftrightarrow\quad\frac{\tr R}{2}\Id-R\in\SS{+}.
\end{equation}
By \eqref{E:m3char} or by Lemma 3.2 of \cite{DS12} we have that
\begin{equation}
\label{E:br0}
 \exists\,r_0>0\text{ s.t. }B_{r_0}(\Id)\subset\mathcal M_3.
\end{equation}

\section{Main perturbation step}
\label{S:mainpertstep}
\begin{definition}
\label{D:ersystem}
 Let $(v,p,\mathring R)\in C^1(\T\times(a_1,a_2);\R^3\times\R\times\SS{0})$. They solve the \emph{Euler-Reynolds system} if they satisfy
\begin{equation}\label{E:ersystem}
 \left\{\begin{aligned}
       &\partial_t v+\div (v\otimes v)+\nabla p=\div \mR\\
&\div v=0  
        \end{aligned}
\right.
\end{equation}
 on $\T\times(a_1,a_2)$. The trace free matrix field $\mR$ is called \emph{Reynolds stress tensor}.
\end{definition}

\begin{proposition}
\label{P_pert_ctheta}
Let $e\in C^{\infty}([a_1,a_2];\R^+)$. Then $\exists\,\eta,\,M>0$ depending on $e$ such that the following holds. Let $(v,p,\mathring R)\in C^{1}(\T\times[a_1,a_2];\R^3\times\R\times \SS{0})$ be a solution of \eqref{E:ersystem}, $0<\beta\leq\frac{1}{2}$, $0<\hat\delta=(\delta')^{\frac{3}{2}}=(\delta'')^{\frac{9}{4}}\leq1$ and $\varphi,\phi\in C^\infty([a_1,a_2];[0,1])$ such that, setting
\begin{align}
 \md&:=\phi^2\delta'+(1-\phi^2)\delta''\label{E:md}\\
\delta^\ast&:=\varphi^2\hat\delta+(1-\varphi^2)\md\label{E:deltaast}
\end{align}
then
\begin{align}
 &\Big|e-\int_{\T}|v|^2-\delta^\ast e\Big|\leq\frac{\beta}{2}\delta^\ast e,\label{E_pert_ctheta_hp_ev}\\
&\|\mathring R-(1-\varphi^2)\mR'\|_0\leq\chi_{\{\varphi>0\}}\eta\hat\delta,\quad\|\mR'\|_0\leq\eta\md\label{E_pert_ctheta_r0}
\end{align}
$\forall\,t\in[a_1,a_2]$, for some $\mR'\in C^{1}(\T\times[a_1,a_2];\SS{0})$.

Let $\bar\delta<\min\Big\{\frac{1}{2}\hat\delta, \hat\delta^{\frac{3}{2}}\Big\}$, $0<\eps<1$, $\bar C>1$, $\psi\in C^\infty([a_1,a_2];[0,1])$.

Then $\exists\,0\leq\ell<\frac{a_2-a_1}{2}$ and $(v_1,p_1,\mathring R_1)\in C^{1}(\T\times[a_1+\ell,a_2-\ell];\R^3\times\R\times \SS{0})$ which solves \eqref{E:ersystem} and satisfies, for all $t\in[a_1+\ell,a_2-\ell]$:
\begin{align}
&v_1=v+\psi w_1\label{E_v_dec}\\
&\mathring R_1=(1-\psi^2)\mathring R+\psi {\mathring R_{1,1}}+\psi'{\mathring R_{1,2}}\label{E_R_dec}\\
&p_1=p+\psi p_o
\end{align}
together with the following estimates
\begin{align}
 &\Big|e-\int_{\T}|v_1|^2-\Big[\psi^2\bar\delta e+(1-\psi^2)\Big(e-\int_{\T}|v|^2\Big)\Big]\Big|\leq\psi\frac{\beta}{2}\bar\delta e\label{E_e_dec}\\
&\|v_1-v\|_0\leq \psi M\sqrt{\delta^\ast}\label{E_wC0}\\
&|||v_1-v|||_1\leq\psi A_1\hat\delta^{\frac{3}{2}}\Big(\frac{D}{\bar\delta^2}\Big)^{1+\eps}\bar\delta^{-\frac{17}{27}-\frac{4}{9}\eps}+|\psi'|\frac{A_1}{\bar C}\sqrt{\delta^\ast}\label{E_wC1}\\
&\|p_1-p\|_0\leq \psi{M^2}\delta^\ast\label{E_pC0}\\
&\|\psi{\mathring R_{1,1}}\|_0\leq\psi\frac{\eta}{2}\bar\delta\label{E_r110}\\
&\|\psi'{\mathring R_{1,2}}\|_0\leq|\psi'|\frac{\eta}{2\bar C}\bar\delta^{\frac{28}{27}+\frac{4}{9}\eps}\label{E_r120}\\
&|||\psi\mR_{1,1}|||_1\leq\psi A_1\hat\delta^{\frac{3}{2}+\eps}\Big(\frac{D}{\bar\delta^2}\Big)^{1+\eps}\bar\delta^{\frac{7}{27}-\frac{4}{9}\eps}+|\psi'|\frac{1}{\bar C}\bar\delta^{\frac{28}{27}}\label{E_r111}\\
&|||\psi'\mR_{1,2}|||_1\leq|\psi'| \frac{A_1}{\bar C}\hat\delta^{\frac{3}{2}+\frac{3}{4}\eps}\Big(\frac{D}{\bar\delta^2}\Big)^{1+\eps}\bar\delta^{\frac{7}{27}+\frac{\eps}{6}}+|\psi''|\frac{A_1}{\bar C^2}\bar\delta^{\frac{28}{27}+\frac{4}{9}\eps}\label{E_r121}
\end{align}
 where $A_1=A_1(\eps,e, \|v\|_{C^0})$ and
\begin{align}
 D&:=\max\Big\{1,\|v\|_{C^1},\|\mathring R\|_{C^1}\Big\}.\label{E_D}
\end{align}
\end{proposition}

\begin{corollary}
\label{C_pert_ctheta}
 Let us assume that $0<\eps<\frac{7}{12}$ and the cut-off function $\psi$ of Proposition \ref{P_pert_ctheta} satisfies
\begin{equation}
\label{E_psi'bounds}
 |\psi'|\leq\bar C(\delta')^{-\frac{3}{4}\eps},\quad|\psi''|\leq\bar C^2(\delta')^{-\frac{3}{2}\eps}.
\end{equation}
Then,
\begin{align}
&|||v_1-v|||_1\leq\chi_{\{\psi>0\}}A_1\hat\delta^{\frac{3}{2}}\Big(\frac{D}{\bar\delta^2}\Big)^{1+\eps}\bar\delta^{-\frac{17}{27}-\frac{4}{9}\eps}\label{E_wC1diff}\\
&\|\mR_1-(1-\psi^2)\mR\|_0\leq\chi_{\{\psi>0\}}\eta\bar\delta\label{E_r120diff}\\
&|||\mR_1-(1-\psi^2)\mR|||_1\leq\chi_{\{\psi>0\}}A_1\hat\delta^{\frac{3}{2}}\Big(\frac{D}{\bar\delta^2}\Big)^{1+\eps}\label{E_r121diff}
\end{align}
If moreover $\varphi\equiv 1$ (or equivalently $\delta^\ast\equiv\hat\delta$) one can replace \eqref{E_wC1diff} with the better estimate
\begin{equation}
 |||v_1-v|||_1\leq\chi_{\{\psi>0\}}A_1\hat\delta^{\frac{3}{2}}\Big(\frac{D}{\bar\delta^2}\Big)^{1+\eps}\label{E_wC1eq}
\end{equation}
\end{corollary}

\begin{remark}

Notice that if both $\varphi=\psi\equiv1$, then the statements of Proposition \ref{P_pert_ctheta} are the same as in Proposition 2.2 of \cite{DS12H}.
\end{remark}

\section{Preliminary geometric and analytic facts}
\label{S:prelimdef}

We define now the linear space of stationary solutions of the Euler equations which will be used to construct the main perturbations to the subsolutions of the Euler equations, namely the so called \emph{Beltrami flows}. We recall the following Proposition from \cite{DS12}.

\begin{proposition}
\label{P:beltrflows}
 Let $\lambda_0\geq1$ and let $A_k\in\R^3$ s.t. 
\[
 A_k\cdot k=0,\quad |A_k|=1,\quad A_{-k}={A_k},
\]
for $k\in\Z^3$ with $|k|=\lambda_0$. Let
\[
 B_k=A_k+i\frac{k}{|k|}\otimes A_k.
\]
Then, for any $a_k\in \C$ such that $a_{-k}=\overline{a_k}$ the vector field
\begin{equation}
 \label{E:bflow}
W(\xi)=\underset{|k|=\lambda_0}{\sum}a_kB_ke^{i k\cdot \xi}
\end{equation}
is real valued as well as its tensor product $W\otimes W$ and it satisfies
\begin{equation*}
 \div W=0,\quad\div(W\otimes W)=\nabla\Big(\frac{|W|^2}{2}\Big),
\end{equation*}
\begin{equation*}
 \fint W\otimes W=\underset{|k|=\lambda_0}{\sum}|a_k|^2\Big(\Id-\frac{k}{|k|}\otimes\frac{k}{|k|}\Big).
\end{equation*}
\end{proposition}

The following geometric Lemma is a refinement made in \cite{Cho} of Lemma 3.2 of \cite{DS12H}.

\begin{lemma}
\label{L:geom1}
 Let $N\geq1$ and $\mathcal N$ an open set s.t. $\overline{\mathcal N}\subset\mathcal M_3$. Then, there exists $\lambda_0>1$, 
 pairwise disjoint and symmetric sets 
\[
\Lambda_j\subset\{k\in\Z^3:\,|k|=\lambda_0\},\quad j=1,\dots,N
\] 
and positive smooth functions
\[
 \gamma^{(j)}_k\in C^{\infty}(\mathcal N),\quad k\in\Lambda_j, \,j=1,\dots,N
\]
such that $\gamma^{(j)}_k=\gamma^{(j)}_{-k}$ and 
\begin{equation*}
R=\underset{k\in\Lambda_j}{\sum}\Big(\gamma^{(j)}_k(R)\Big)^2\Big(\Id-\frac{k}{|k|}\otimes\frac{k}{|k|}\Big), \quad\forall\,R\in\mathcal N,\,j=1,\dots,N.
\end{equation*}
\end{lemma}

While constructing the perturbations one also needs (as in Section 4.1 of \cite{DS12}) to introduce suitable partitions of unity on $\R^3$ in order to discretize the space of velocities. These partitions of unity depend on a integer parameter $\mu$ which, roughly speaking, as it increases improves the accuracy of the discretization.

Let $\CC_j$, $j=1,\dots,8$ be the equivalence classes of $\Z^3$ w.r.t. the equivalence relation $l-l'\in(2\Z)^3$, and let $\{a_k\}_{k\in\Z^3}\subset C^{\infty}(\R^3;[0,1])$ be a smooth partition of unity of $\R^3$ --namely, $\sum_k(a_k(v))^2\equiv 1$-- s.t. $\supp a_k\subset\inter\, B_{1}(k)$. For $\mu\in\N$, $k\in\Z^3$ and $j=1,\dots,8$ define then the functions
\[
 \phi^{(j)}_{k,\mu}(v,\tau):=\sum_{l\in\mathcal C_j}\alpha_l(\mu v)e^{-i\frac{k\cdot l}{\mu}\tau},\quad v\in\R^3,\,\tau\in[0,+\infty).
\]

We report in the following proposition the derivative estimates on the functions $\phi^{(j)}_k=\phi^{(j)}_k(v,\tau)$ given in Proposition 4.2 of \cite{DS12H}, where $v$ is considered as an independent variable in $\R^3$.
Let us first introduce the seminorms 
\[
 [\cdot]_{m,R}=[\cdot]_{C^r(B_R(0)\times[a_1+\ell,a_2-\ell])}.
\]

\begin{proposition}
\label{P_phijk}
 There are constants $C$ depending only on $m\in\N$ such that the following estimates hold
\begin{align}
 [\phi^{(j)}_{k,\mu}]_{m,R}+R^{-1}[\partial_\tau\phi^{(j)}_{k,\mu}]_{m,R}+R^{-2}[\partial^2_\tau\phi^{(j)}_{k,\mu}]_{m,R}&\leq C\mu^m\label{E_phijk1}\\
[\partial_\tau\phi^{(j)}_{k,\mu}+i(k\cdot v)\phi^{(j)}_{k,\mu}]_{m,R}&\leq C\mu^{m-1}\label{E_phijk2}\\
R^{-1}[\partial_\tau(\partial_\tau\phi^{(j)}_{k,\mu}+i(k\cdot v)\phi^{(j)}_{k,\mu})]_{m,R}&\leq C\mu^{m-1}.\label{E_phijk3}
\end{align}
\end{proposition}

Then we recall the elliptic operators which are used in \cite{DS12} and \cite{DS12H} to define the corrector terms to the main perturbations of a given subsolution.

\begin{definition}[Leray projector]
\label{D:lerayproj}
For any $v\in C^\infty(\T;\R^3)$, set
\[
 \mathcal Qv:=\nabla \phi+\fint_{\T} v,
\]
where $\phi\in C^{\infty}(\T)$ is the solution to $\nabla \phi=\div v$ in $\T$ with $\fint_{\T}\phi=0$. We denote by $\mathcal P=\Id-\mathcal Q$ the \emph{Leray projector} onto divergence-free vector fields with zero average.
\end{definition}

\begin{definition}[The operator $\mathcal R$]
\label{D:roperator}
 For any $v\in C^\infty(\T;\R^3)$ we define $\mathcal Rv\in C^\infty(\T;\SS{})$ as
\begin{equation*}
 \mathcal Rv=\frac{1}{2}\Big(\nabla\mathcal P u +(\nabla\mathcal P u)^T\Big)+\frac{3}{2}\Big(\nabla u +(\nabla u)^T\Big)-\frac{1}{2}(\div u)\Id,
\end{equation*}
where $u\in C^\infty(\T;\R^3)$ is the solution to
\[
 \triangle u=v-\fint_{\T}v, \quad\fint_{\T}u=0.
\]
\end{definition}

By the following lemma, the operator $\mathcal R$ acts on divergence free vector fields as an inverse of the divergence operator.
\begin{lemma}[\cite{DS12}, Lemma 4.3] 
\label{L_rop}
For any $v\in C^\infty(\T;\R^3)$, $\mathcal Rv\in \SS{0}$ and $\div\mathcal Rv=v-\fint_{\T} v$.  
\end{lemma}

Finally, we state the following simple Lemma, which will be used in the proof of the main perturbation step (Proposition \ref{P_pert_ctheta}).

\begin{lemma}
\label{L:bound}
 Let $a,b\in(0,1]$, $a\leq b$. Then, the function $g:[0,1]\to\R$ given by
\[
 g(s)=\frac{sa+(1-s^2)b}{s^2a+(1-s^2)b}
\]
satisfies the bounds
\[
 1\leq g(s)\leq\frac{5}{4}.
\]
\end{lemma}

\begin{proof}
 One can easily check that $g\geq1=g(0)=g(1)$ and that $\underset{[0,1]}{\max}\,g$ is attained at the point $f(a,b):=\frac{\sqrt{b}}{\sqrt{b}+\sqrt{a}}$ and is given by 
\[
 g(f(a,b))=1+\frac{a}{2\sqrt{b}(\sqrt{a}+\sqrt{b})}.
\]
 Finally, $\underset{a\leq b}{\max}\,g(f(a,b))=g(f(b,b))=\frac{5}{4}$.
\end{proof}

We now recall some preliminary H\"older and Schauder estimates from \cite{DS12H}.
As in \cite{DS12H} we will denote the constants which appear in the estimates with the letter $C$, eventually adding some subscripts according to the following rules.
\begin{itemize}
 \item $C$ without subscripts denote universal constants;\\
\item $C_h$ denote constants appearing in standard H\"older inequalities in spaces $C^r$. The dependence on $r$ is omitted, since $r\geq0$ will be fixed, even though quite large, at the end of the construction;\\
\item $C_e$ are constants which depend on the upper and lower bounds for $e$;\\
\item $C_v$ constants which may depend not only on the upper and lower bounds for $e$ but also on the supremum norm of $v$, i.e. $\|v\|_{C^0}$;\\
\item $C_s$, $C_{e,s}$, $C_{v,s}$ are constants involved in Schauder estimates for $C^{m+\alpha}$-norms of elliptic operators, which usually degenerate as $\alpha$ tends to $0$ or $1$. The ones denoted by $C_{e,s}$, $C_{v,s}$ depend also respectively on the upper and lower bounds for $e$ and on $\|v\|_{C^0}$.
\end{itemize}
The above constants might also depend on $\eps$ or $\omega$ but never on $\mu,\,\lambda,\,\ell,\,\psi,\,\phi,\,\varphi,\,\hat\delta,\,\bar\delta,\,\delta',\,\delta''$ and $D$.

We recall from \cite{DS12} the following elliptic estimates for the operators defined in \ref{D:lerayproj} and \ref{D:roperator}.
\begin{proposition}[Schauder estimates]
\label{P_schauder}
 For any $\alpha\in (0,1)$ and any $m\in\N$ there exists a constant $C_s$ such that
\begin{align}
 \|\mathcal Qv\|_{m+\alpha}\leq C_s\|v\|_{m+\alpha};\notag\\
 \|\mathcal Pv\|_{m+\alpha}\leq C_s\|v\|_{m+\alpha};\notag\\
 \|\mathcal Rv\|_{m+1+\alpha}\leq C_s\|v\|_{m+\alpha};\notag\\
 \|\mathcal R(\div A)\|_{m+\alpha}\leq C_s\|A\|_{m+\alpha};\label{E_Rdiv}\\
 \|\mathcal R\mathcal Q(\div A)\|_{m+\alpha}\leq C_s\|A\|_{m+\alpha}.\notag
\end{align}
\end{proposition}

When the above operators are applied to the product of smooth vector fields and highly oscillating trigonometric functions we get the following estimates (see Propositions 5.2 in \cite{DS12} and 4.4. in \cite{DS12H})
\begin{proposition}
\label{P_osc_est}
Let $k\in \Z^3\setminus \{0\}$ and $\lambda \geq 1$. Then, for any $a\in C^\infty(\T)$ and $m\in \N$ we have
\begin{equation}
\label{E_osc_est_1}
 \Big|\int a(x) e^{i\lambda k\cdot x}\,dx\Big|\leq\frac{[a]_m}{\lambda^m}.
\end{equation}
For any $F\in C^\infty(\T;\R^3)$, let $F_\lambda:=F(x) e^{i\lambda k\cdot x}$. Then we have
\begin{align*}
 \|\mathcal R(F_\lambda)\|_{\alpha}\leq \frac{C_s}{\lambda^{1-\alpha}}\|F\|_0+\frac{C_s}{\lambda^{m-\alpha}}[F]_m+\frac{C_s}{\lambda^{m}}[F]_{m+\alpha},\\
\|\mathcal R\mathcal Q(F_\lambda)\|_{\alpha}\leq \frac{C_s}{\lambda^{1-\alpha}}\|F\|_0+\frac{C_s}{\lambda^{m-\alpha}}[F]_m+\frac{C_s}{\lambda^{m}}[F]_{m+\alpha}.
\end{align*}
 \end{proposition}

\section{Construction of the maps $v_{1}$, $p_{1}$ and $\mathring R_{1}$}
\label{S_construction_step}

As in \cite{DS12} and \cite{DS12H}, the idea to construct $v_1$ is to perturb $v$ with a fast oscillating time-dependent ``patching'' of Beltrami flows $w_o$ and then adding a smaller perturbation term $w_c$ so that $w_1=w_o+w_c$ satisfies the divergence free constraint. However, in this case both perturbations are multiplied by the cut-off function $\psi$. Also the perturbation terms of the pressure will be modified accordingly, and the Reynolds stress tensor $\mathring R_1$ will be defined as in \cite{DS12H} using the elliptic ``$\mathcal R$ operator'' (see Definition \ref{D:roperator}).

Given $(v,p,\mathring R),\,\varphi,\,\phi,\,\hat\delta,\,\delta',\,\delta''$ as in Proposition \ref{P_pert_ctheta}, we fix $\eps,\,\psi,\,\bar C,\, \bar\delta,\,\ell$ as in the main statement, with $\ell$ satisfying also
\begin{equation}
\label{E:Dell}
D\ell\leq\eta{\bar\delta}. 
\end{equation}
The perturbations will also depend on two parameters $\lambda, \mu$ s.t.
\begin{equation}
 \label{E_lm_int}
\lambda,\,\mu,\,\frac{\lambda}{\mu}\,\in\N.
\end{equation}
 In order to simplify calculations, we assume from now onwards the following inequalities, which will be consistent with the choice of parameters we make in Section \ref{Ss:pert_ctheta} to get the estimates of Proposition \ref{P_pert_ctheta}.
\begin{equation}
\label{E_const_ineq}
 \mu\geq(\delta^\ast)^{-1}\geq1,\quad\lambda\geq\mathrm{max}\biggl\{(\mu D)^{1+\omega},\ell^{-(1+\omega)}\biggr\},
\end{equation}
where 
\begin{equation}
 \label{E_omeps}
\omega=\frac{\eps}{2+\eps}.
\end{equation}
%so that $\frac{1+\omega}{1-\omega}=1+\eps$.

\subsection{Mollifications}
Let $\chi\in C^\infty(\R^3\times\R)$ be a nonnegative radial kernel supported in $[-1,1]^4$ and define 

\begin{align*}
\chi_\ell(x,t)&=\frac{1}{\ell^4}\chi\Big(\frac{x}{\ell},\frac{t}{\ell}\Big),\\
v_\ell(x,t)&=\int_{\T\times[-1,1]}v(x-y,t-s)\chi_\ell(y,s)\,dy\,ds,\\
\mathring R_\ell(t,x)&=\int_{\T\times[-1,1]}\mathring R(x-y,t-s)\chi_\ell(y,s)\,dy\,ds.
\end{align*}
Notice that $v_\ell$ and $\mathring R_\ell$ are well defined on $\T\times[a_1+\ell,a_2-\ell]$. That is why we require $\ell<\frac{a_2-a_1}{2}$.
Moreover
\begin{align}
\|v_\ell-v\|_0+\|\mathring R_\ell-\mathring R\|_0&\leq CD\ell,\label{E_vrell_0}\\
|||v_\ell|||_r+|||\mathring R_\ell|||_r&\leq C(r)D \ell^{1-r}, \quad r\geq1,\label{E_vrell_r}
\end{align}
for all $t\in[a_1+\ell,a_2-\ell]$.

As a consequence, since $||v_\ell|^2-|v|^2|\leq|v-v_\ell|^2+2|v-v_\ell||v|$ and \eqref{E:Dell} holds
\begin{align}
 \int_{\T}||v_\ell|^2-|v|^2|&\leq C(D\ell)^2+CD\ell \sqrt{e}\notag\\
&\leq C\eta\bar\delta(\underset{s\in[a_1,a_2]}{\max}e(s)^{1\slash2}+1).\label{E_vvell_int}
\end{align}

\subsection{Construction of $v_1$}
We define $v_1$ adding to $v$ two perturbations as in \cite{DS12} and \cite{DS12H}, but localized by the cut-off function $\psi$, namely
\begin{equation}
\label{E_v1_ctheta}
 v_{1}=v+\psi w_o+\psi w_c, 
\end{equation}
with $w_o$ and $w_c$ defined as in Section 3 of \cite{DS12H}. Then we set 
\[
 w_1:=w_o+w_c.
\]

The main perturbation term is $w_o$, while $w_c$ is a corrector term that makes $v_1$ divergence free and it is given by 
\begin{equation}
 \label{E_wc_def}
w_c=-\mathcal Qw_o,
\end{equation}
where $\mathcal Q=\Id-\mathcal P$ and $\mathcal P$ is the Leray projection operator as in  Definition \ref{D:lerayproj}.

%W.l.o.g., we assume in the following that $e$ and $(v,p,\mathring R)$ are defined and satisfies the assumptions of Proposition \ref{P_pert_ctheta} on a slightly larger domain $(-h,1+h)\times\T$, for some $h>0$. 

\subsubsection{Construction of $w_o$}
Define, for $(x,t)\in\T\times[a_1+\ell,a_2-\ell]$
\begin{align*}
\rho_\ell(t)&=\frac{1}{3(2\pi)^3}\Big(e(t)(1-\bar\delta)-\int_{\T}|v_\ell|^2(x,t)\,dx\Big),\\
R_\ell(x,t)&=\rho_\ell(t)\Id-\mathring R_\ell(x,t).
\end{align*}
Provided $\frac{R_\ell}{\rho_\ell}\in\mathcal N$, with $\mathcal N$ satisfying the assumptions of Lemma \ref{L:geom1}, the main perturbation $w_o$ on $\T\times[a_1+\ell,a_2-\ell]$ can be defined as
\begin{equation}
\label{E_w_o_ctheta}
 w_o(x,t):=W_o(x,t,\lambda t,\lambda x)
\end{equation}
where
\begin{equation*}
W_o(y,s,\tau,\xi):=\sqrt{\rho_\ell(s)}\sum_{j=1}^8\sum_{k\in\Lambda_j}\gamma^{(j)}_k\Big(\frac{R_\ell(y,s)}{\rho_\ell(s)}\Big)\phi_{k,\mu}^{(j)}(v_\ell(y,s),\tau) B_k e^{ik\cdot\xi},
\end{equation*}
$\gamma^{(j)}_k,\,\phi^{(j)}_{k,\mu}$ are defined in Section \ref{S:prelimdef} and $B_k\in \mathbb C^3$ are unit vectors satisfying the assumptions of Proposition \ref{P:beltrflows}.

\begin{lemma}
\label{L_etaw_o}
 If $\eta\leq\eta(e)$, then $\frac{R_\ell}{\rho_\ell}\in B_{r_0}(\Id)$, with $r_0$ satisfying \eqref{E:br0}. 
\end{lemma}

\begin{proof}
By \eqref{E_pert_ctheta_r0}, \eqref{E_vrell_0} and \eqref{E:Dell} 
\begin{align*}
 \Big\|\frac{R_\ell}{\rho_\ell}-\Id\Big\|_0&=\Big\|\frac{\mR_\ell}{\rho_\ell}\Big\|_0\leq\frac{CD\ell+\|\mR\|_0}{\underset{[a_1+\ell,a_2-\ell]}{\min}\rho_\ell}\notag\\
&{\leq}\frac{C\eta\bar\delta+\chi_{\{\varphi>0\}}\eta\hat\delta+(1-\varphi^2)\eta\md}{\underset{[a_1+\ell,a_2-\ell]}{\min}\rho_\ell}\notag\\
&\leq\frac{3C\eta\delta^\ast}{\underset{[a_1+\ell,a_2-\ell]}{\min}\rho_\ell}.
\end{align*}
Since $\bar\delta\leq\frac{1}{2}{\delta^\ast} $
\begin{align}
 \rho_\ell(t)&\geq\frac{1}{3(2\pi)^3}\Big\{e\Big(1-\frac{\delta^\ast}{2}\Big)-\int_{\T}|v|^2-\int_{\T}||v_\ell|^2-|v|^2|\Big\}\notag\\
&\overset{\eqref{E_pert_ctheta_hp_ev}\,\eqref{E_vvell_int}}{\geq}\frac{1}{3(2\pi)^3}\Big\{\Big(\frac{1}{2}-\frac{\beta}{2}\Big)\delta^\ast e-C\eta\bar\delta(\underset{[a_1+\ell,a_2-\ell]}{\max}\sqrt{e}+1)\Big\}.\notag
\end{align}
and since $\beta\leq\frac{1}{2}$, if $\eta=\eta(e)$ is sufficiently small we get the lower bound
\begin{align}
\label{E_rhoellgeq}
 \rho_\ell(t)&\geq\frac{C'}{3(2\pi)^3}\delta^\ast
\end{align}
for some constant $C'>0$.
Hence, provided $\eta<\frac{C'}{9(2\pi)^3C}$, the lemma is proved.
\end{proof}

Finally notice that, analogously to \eqref{E_rhoellgeq}, $\rho_\ell$ can be estimated from above as
\begin{align}
 \rho_\ell&\leq e-\int|v|^2+\Big|\int|v_\ell|^2-|v|^2\Big|\notag\\
&\overset{\eqref{E_pert_ctheta_hp_ev}}{\leq}\delta^\ast e+\frac{\beta}{2}\delta^\ast e+C\eta\bar\delta\bigl(1+\underset{[a_1,a_2]}{\max}\sqrt{e}\bigr)\notag\\
&\leq\frac{5}{4}\delta^\ast e+C\eta\bar\delta\bigl(1+\underset{[a_1,a_2]}{\max}\sqrt{e}\bigr)\notag\\
&\leq C'\delta^\ast(1+\underset{[a_1,a_2]}{\max}e).\label{E:rho_above}
\end{align}
Then, since $|w_o(x,t)|\leq C\sqrt{\rho_\ell(t)}$, we can choose $M=M(e)>1$ s.t.
\begin{equation}
 \label{E:M}
\|w_o\|_0\leq\frac{M}{2}\sqrt{\delta^\ast}.
\end{equation}

\subsection{Definition of $p_{1}$}
We define the pressure $p_{1}$ as
\begin{equation*}
 p_{1}:=p-\psi^2\frac{|w_o|^2}{2}-\psi\frac{2\langle(v-v_\ell),w_1\rangle}{3}.
\end{equation*}
By \eqref{E:M} and \eqref{E_vrell_0}, provided $\eta$ is sufficiently small we have 
\begin{align}
\|p_1-p\|_0&\leq \psi^2\frac{M^2}{4}\delta^\ast+C\eta\bar\delta\|w_1\|_0\notag\\
&\leq\psi^2\frac{M^2}{4}\delta^\ast+\frac{3}{4}\bar\delta\|w_1\|_0\label{E:pc0}
\end{align}

\subsection{Definition of $\mR_{1}$}

According to our definitions of $v_1$ and $p_1$, it is fairly easy to see that
\begin{align*}
 \partial_t v_1+\div(v_1\otimes v_1)+\nabla p_1&=\div\bigl[(1-\psi^2)\mR\bigr]\\
&+\partial_t(\psi w_1)+\div(\psi w_1\otimes v_\ell+v_\ell\otimes \psi w_1)\\
&+\div(\psi^2w_1\otimes w_1-\frac{\psi^2}{2}|w_o|^2\Id+\psi^2\mR_\ell)\\
&+\div\bigl[\psi^2(\mR-\mR_\ell)\bigr]\\
&+\div\Big(\psi w_1\otimes(v-v_\ell)+(v-v_\ell)\otimes \psi w_1-\psi\frac{2\langle v-v_\ell,w_1\rangle}{3}\Id\Big)
\end{align*}
Hence, recalling the Definition \ref{D:roperator} of the operator $\mathcal R$, we set 
\begin{align}
 \mR_1&:=(1-\psi^2)\mR\notag\\
&+\mathcal R\bigl[\partial_t(\psi w_1)+\div(\psi w_1\otimes v_\ell+v_\ell\otimes \psi w_1)\bigr]\notag\\
&+\mathcal R\bigl[\div(\psi^2w_1\otimes w_1-\frac{\psi^2}{2}|w_o|^2\Id+\psi^2\mR_\ell)\bigr]\notag\\
&+\psi^2(\mR-\mR_\ell\notag)\\
&+\psi w_1\otimes(v-v_\ell)+(v-v_\ell)\otimes \psi w_1-\psi\frac{2\langle v-v_\ell,w_1\rangle}{3}\Id\label{E_def_reyn}.
\end{align}
By Lemma \ref{L_rop} and the fact that $\partial_tv_1+\div(v_1\otimes v_1)+\nabla p_1$ has zero average, it follows that $(v_1,p_1,\mR_1)$ satisfies \eqref{E:ersystem}.

\section{Proof of the main perturbation step}
\subsection{Doubling the variables and corresponding estimates}
\label{S_doubling}
As in \cite{DS12H}, for any $k\in\Lambda_j\subset\{|k|=\lambda_0\}$ we set
\begin{equation}
\label{E:ak}
a_k(y,s,\tau)=\sqrt{\rho_\ell(s)}\gamma^{(j)}_k\Big(\frac{R'_\ell(y,s)}{\rho_\ell(s)}\Big)\phi^{(j)}_{k,\mu}(v(y,s),\tau).
\end{equation}
and
\[
 \psi(s)W_o(y,s,\tau,\xi)=\psi(s)\underset{j=1}{\overset{8}{\sum}}\underset{k\in\Lambda_j}{\sum}a_k(y,s,\tau)B_ke^{ik\cdot\xi}
\]
so that our main perturbation term is given by 
\[
\psi(t) w_o(x,t)=\psi(t)W_o(x,t, \lambda t,\lambda x).
\]

The next proposition is the analogue of Proposition 5.1 in \cite{DS12H}. The following estimates on the derivatives of the coefficients of the main perturbation term will be the basis for estimating also the derivatives of the other correction terms in the supremum norm.

\begin{proposition}
\label{P_doubling_ctheta}
Let $a_{k}\in C^\infty(\T\times[a_1+\ell,a_2-\ell]\times\R)$ be as in \eqref{E:ak}, $\psi$ the cut-off function of Proposition \ref{P_pert_ctheta} and $D$ given by \eqref{E_D}. 

Then, for any $r\geq1$ and any $\alpha\in[0,1]$ we have the following estimates
\begin{align}
 \|\psi a_{k}(\cdot,s,\tau)\|_r&\leq\psi(s)C_e\sqrt{{\delta^\ast} }(\mu^rD^r+\mu D\ell^{1-r})\label{E_62}\\
 \|\partial_\tau \psi a_{k}(\cdot,s,\tau)\|_r+\|\partial^2_\tau \psi a_{k}(\cdot,s,\tau)\|_r&\leq\psi(s)C_e\sqrt{{\delta^\ast} }(\mu^rD^r+\mu D\ell^{1-r})\label{E_63}\\
\|(\partial_\tau \psi a_{k}+ i(k\cdot v_\ell)\psi a_{k})(\cdot,s,\tau)\|_r&\leq\psi(s)C_e\sqrt{{\delta^\ast} }(\mu^{r-1}D^r+D\ell^{1-r})\label{E_64}\\
\|\partial_\tau(\partial_\tau \psi a_{k}+ i(k\cdot v_\ell)\psi a_{k})(\cdot,s,\tau)\|_r&\leq\psi(s)C_e\sqrt{{\delta^\ast} }(\mu^{r-1}D^r+D\ell^{1-r})\label{E_65}
\end{align}
\begin{align}
 \|\psi a_{k}(\cdot,s,\tau)\|_\alpha&\leq\psi(s)C_e\sqrt{{\delta^\ast} }\mu^\alpha D^\alpha\label{E_66}\\\
 \|\partial_\tau \psi a_{k}(\cdot,s,\tau)\|_\alpha+\|\partial^2_\tau a_{k,\psi}(\cdot,s,\tau)\|_\alpha&\leq\psi(s)C_e\sqrt{{\delta^\ast} }\mu^\alpha D^\alpha\label{E_67}\\
\|(\partial_\tau \psi a_{k}+ i(k\cdot v_\ell)\psi a_{k})(\cdot,s,\tau)\|_\alpha&\leq\psi(s)C_e\sqrt{{\delta^\ast} }\mu^{\alpha-1}D^\alpha\label{E_68}\\
\|\partial_\tau(\partial_\tau \psi a_{k}+ i(k\cdot v_\ell)\psi a_{k})(\cdot,s,\tau)\|_\alpha&\leq\psi(s)C_e\sqrt{{\delta^\ast} }\mu^{\alpha-1}D^\alpha\label{E_69}
\end{align}
Moreover, for any $r\geq0$ 
\begin{align}
&\|\partial_s\psi a_{k}(\cdot,s,\tau)\|_r\leq \psi(s)C_e\sqrt{{\delta^\ast} }(\mu^{r+1}D^{r+1}+\mu D\ell^{-r})+|\psi'|(s)C_e\sqrt{{\delta^\ast} }(\mu^r D^r+\mu D\ell^{1-r})\label{E_70}\\
 &\|\partial^2_{s\tau} \psi a_{k}(\cdot,s,\tau)\|_r\leq\psi(s) C_e\sqrt{{\delta^\ast} }(\mu^{r+1}D^{r+1}+\mu D\ell^{-r})+|\psi'|(s)C_e\sqrt{{\delta^\ast} }(\mu^r D^r+\mu D\ell^{1-r})\label{E_71}\\
&\|\partial^2_{s} \psi a_{k}(\cdot,s,\tau)\|_r\leq \psi(s)C_e\sqrt{{\delta^\ast} }(\mu^{r+2}D^{r+2}+\mu D\ell^{-1-r})+|\psi'|(s)C_e\sqrt{{\delta^\ast} }(\mu^{r+1}D^{r+1}+\mu D\ell^{-r})\notag\\
&\qquad\qquad\qquad\quad+|\psi''|(s)C_e\sqrt{{\delta^\ast} }(\mu^r D^r+\mu D\ell^{1-r})\label{E_72}\\
&\|\partial_s(\partial_\tau (\psi a_{k}+ i(k\cdot v_\ell)\psi a_{k})(\cdot,s,\tau)\|_r\leq \psi(s)C_e\sqrt{{\delta^\ast} }(\mu^{r}D^{r+1}+\mu D\ell^{-r})+|\psi'|(s)C_e\sqrt{{\delta^\ast} }(\mu^{r-1}D^r+\mu D\ell^{1-r}).\label{E_73}
\end{align}
We also have
\begin{equation}
\label{E_double_W2}
       \psi W_o\otimes \psi W_o(y,s,\tau,\xi)= \psi^2(s)R_\ell(y,s)+\psi^2(s)\sum_{1\leq |k|\leq2\lambda_0}U_{k}(y,s,\tau)e^{ik\cdot\xi},
      \end{equation}
where $U_k\in C^\infty (\T\times[a_1+\ell,a_2-\ell]\times\R)$ satisfy 
\begin{equation}
 U_kk=\frac{1}{2}(\tr U_k)k
\end{equation}
 and the following estimates for any $r\geq0$ and any $\alpha\in[0,1]$:
\begin{align}
 \|\psi^2 U_{k}(\cdot,s,\tau)\|_r&\leq \psi^2(s)C_e{\delta^\ast} (\mu^r D^r+\mu D\ell^{1-r})\label{E_ukr}\\
 \|\partial_\tau \psi^2 U_{k}(\cdot,s,\tau)\|_r&\leq \psi^2(s)C_v{\delta^\ast} (\mu^r D^r+\mu D\ell^{1-r})\\
 \|\psi^2 U_{k}(\cdot,s,\tau)\|_\alpha&\leq\psi^2(s) C_e{\delta^\ast} \mu^\alpha D^\alpha\\
 \|\partial_\tau \psi^2 U_{k}(\cdot,s,\tau)\|_\alpha&\leq\psi^2(s) C_v{\delta^\ast} \mu^\alpha D^\alpha.
\end{align}
Moreover, for any $r\geq0$
\begin{equation}
\label{E_ukr_last}
  \|\partial_s \psi^2 U_{k}(\cdot,s,\tau)\|_r\leq\psi^2(s) C_e{\delta^\ast} (\mu^{r+1} D^{r+1}+\mu D\ell^{-r})+\psi(s)|\psi'|(s)C_e{\delta^\ast} (\mu^r D^r+\mu D\ell^{1-r}).
\end{equation}
\end{proposition}

\begin{proof}
Notice that the above estimates are the same as in Proposition 5.1 of \cite{DS12H}, up to replacing $\delta$ with ${\delta^\ast} $ everywhere and multiplying them by suitable powers of $\psi$ and its derivatives.

As in \cite{DS12H} we set
\begin{equation}
\label{E_akgammapsi}
 \psi(s)a_k(y,s,\tau)=\psi(s)\sqrt{\rho_\ell(s)}\Gamma(y,s)\Phi(y,s,\tau),
\end{equation}
 where $\Gamma(y,s)=\gamma^{(j)}_k\Big(\frac{R_\ell(y,s)}{\rho_\ell(s)}\Big)$, $\Phi(y,s,\tau)=\phi^{(j)}_{k,\mu}(v_\ell(y,s),\tau)$.

By definition of $U_k$, it is fairly easy to see that the estimates \eqref{E_ukr}-\eqref{E_ukr_last} are a consequence of \eqref{E_62}-\eqref{E_65} and \eqref{E_70}.

First of all, notice that in the estimates \eqref{E_62}-\eqref{E_69}, since $\psi$ does not depend on $(y,\tau)$, we simply have to multiply the estimates for $a_k$ by the function $\psi$. Analogously, in the estimates \eqref{E_70}-\eqref{E_73} we simply apply the Leibniz rule to the derivation w.r.t. $s$ of the products $\psi a_k$, $\psi \partial_\tau a_k$ etc. and then bring back to the estimates \eqref{E_62}-\eqref{E_69}, with one (or two) order of derivation less w.r.t. $y$ for the terms in which $\psi$ is differentiated. 

Then, we can reduce to study $a_k$ and see why the the only change in the estimates is the replacement of the $\delta$ of \cite{DS12H} by ${\delta^\ast} $. 
First notice that, since the only change in the definition of $a_k$ is the choice of $\rho_\ell$, after applying the Leibniz rule to \eqref{E_akgammapsi} the only terms that might be changing are those containing derivatives of functions of $\rho_\ell$. In particular, it is fairly easy to see that such terms are 
\[
 |\sqrt{\rho_\ell}|,\quad|\partial_s^r\sqrt{\rho_\ell}|,\quad\|\Gamma\|_r,\quad\|\partial_s\Gamma\|_r.
\]
 By \eqref{E:rho_above}, we already know that $|\sqrt{\rho_\ell}|\leq C_e\sqrt{{\delta^\ast} }$, while in \cite{DS12H} one has $|\sqrt{\rho_\ell}|\leq C_e\sqrt{\delta}$.
Moreover,
\begin{align}
 \|\Gamma\|_r&\leq C\Big\|\frac{\mR_\ell}{\rho_\ell}\Big\|_r\leq C\frac{\|\mR_\ell\|_r}{\underset{[a_1+\ell,a_2-\ell]}{\min}\rho_\ell}\notag\\
&\overset{\eqref{E_rhoellgeq}}{\leq} C'\frac{D\ell^{1-r}}{{\delta^\ast} }\notag\\
&\overset{\eqref{E_const_ineq}}{\leq}C'\mu D\ell^{1-r}\notag
\end{align}
and
\begin{align}
 \|\partial_s\Gamma\|_r&\leq C\frac{\|\partial_s\mR_\ell\|_r}{\underset{[a_1+\ell,a_2-\ell]}{\min}\rho_\ell}+C\|\mR_\ell\|_r|\partial_s(\rho_\ell^{-1})|\notag\\
&\leq C\frac{D\ell^{-r}}{{\delta^\ast} }+CD\ell^{1-r}D{\delta^\ast} ^{-2}\notag\\
&\leq C(\mu D\ell^{-r}+\mu^2D^2\ell^{1-r})\notag
\end{align}
exactly as in \cite{DS12H}. 
Finally, 
\begin{align}
 |\partial_s^r\sqrt{\rho_\ell}|&\leq C_h\underset{i=1}{\overset{r}{\sum}}C_e(\delta^\ast)^{\frac{1}{2}-i}|\rho_\ell|^{i-1}|\partial_s^r\rho_\ell|\notag\\
&\leq{\delta^\ast}^{-\frac{1}{2}}|\partial_s^r\rho_\ell|\leq C{\delta^\ast} ^{-\frac{1}{2}}D\ell^{1-r}\notag\\
&\leq C\sqrt{{\delta^\ast} }\mu D\ell^{1-r}\label{E_sqrtrho}
\end{align}
in contrast with $|\partial_s^r\sqrt{\rho_\ell}|\leq C\sqrt{\delta}\mu D\ell^{1-r}$ of \cite{DS12H}.

Since by the Leibniz rule 
\begin{equation*}
 \|\partial^m_{s}\partial^p_{\tau}a_k\|_r\leq C\underset{i=0}{\overset{m}{\sum}}|\partial^i_s\sqrt{\rho_\ell}|\|\partial^{m-i}_s\bigl(\Gamma(y,s)\partial_\tau^p\Phi(y,s,\tau)\bigr)\|_r,
\end{equation*}
where $\|\partial^{m-i}_s\bigl(\Gamma(y,s)\partial_\tau^p\Phi(y,s,\tau)\bigr)\|_r$ can be estimated as in \cite{DS12H}, by \eqref{E_sqrtrho} Proposition \ref{P_doubling_ctheta} is proved.
\end{proof}

\subsection{Estimates on $v_1$}
\label{S_v_1_est}

The estimates for the perturbation term $\psi w_1$ are obtained as in Section 6 of \cite{DS12H}, using the structure of $w_o$, the Schauder estimates for $\mathcal Qw_o$ of Proposition \ref{P_schauder}, and the estimates of Proposition \ref{P_doubling_ctheta}. 

Therefore it is fairly easy to see that the following proposition is the counterpart of Proposition 6.1 of \cite{DS12H}.
\begin{proposition}
\label{P_w0wc_est}
The following estimates hold for all $r\geq0$, $t\in[a_1+\ell,a_2-\ell]$:
\begin{align}
 \|\psi w_o\|_r&\leq \psi C_e\sqrt{{\delta^\ast} }\lambda^r\label{E_w0Cr}\\
\|\partial_t(\psi w_o)\|_r&\leq \psi C_{v}\sqrt{{\delta^\ast} }\lambda^{r+1}+|\psi'|C_{v}\sqrt{{\delta^\ast} }\lambda^r\label{E_w0Crt}
\end{align}
and the following for any $r>0$ which is not integer:
\begin{align}
  \|\psi w_c\|_r&\leq \psi C_{e,s}\sqrt{{\delta^\ast} }D\mu\lambda^{r-1}\label{E_wcCr}\\
 \|\partial_t(\psi w_c)\|_r&\leq \psi C_{v,s}\sqrt{{\delta^\ast} }D\mu\lambda^{r}+|\psi'|C_{v,s}\sqrt{{\delta^\ast} }D\mu\lambda^{r-1}.\label{E_wcCrt}
\end{align}
In particular,
\begin{align*}
\|v_1-v\|_0&= \|\psi w\|_0\leq \psi C_e\sqrt{{\delta^\ast} },\\
\|\partial_t(v_1-v)\|_0+\|v_1-v\|_1&=\|\partial_t(\psi w)\|_0+\|\psi w\|_1\leq \psi C_e\sqrt{{\delta^\ast} }\lambda +|\psi'|C_e\sqrt{{\delta^\ast} }.
\end{align*}
\end{proposition}

\subsection{Estimates on the energy}
\label{S_energy_est}
Analogously to Proposition 7.1 of \cite{DS12H} we obtain the following
\begin{proposition}[]
\label{P_energy_est}
 For any $\alpha\in \bigl(0,\frac{\omega}{1+\omega}\bigr)$ there is a constant $C_{v,s}$, depending only on $\alpha$, $e$ and $\|v\|_{C^0}$ such that, $\forall\,t\in[a_1+\ell,a_2-\ell]$
\begin{equation*}
 \Big|e-\int_{\T}|v_{1}|^2-\Big[\psi^2\bar\delta e +(1-\psi^2)\Big(e-\int_{\T}|v|^2\Big)\Big]\Big|\leq \psi^2 C_e D\ell+ \psi C_{v,s}\sqrt{{\delta^\ast} }\mu D\lambda^{\alpha-1}
\end{equation*}
\end{proposition}

\begin{proof}
  Writing $|v_{1}|^2$ as 
\begin{equation*}
 |v_{1}|^2=|v|^2+\psi^2| w_o|^2+\psi^2|w_c|^2+2\psi v\cdot  w_o+2\psi v\cdot w_c+2\psi^2 w_o\cdot w_c,
\end{equation*}
we get
\begin{align}
 \Big|\int_{\T}|v_{1}|^2-|v|^2-\psi^2| w_o|^2\Big|&\leq C_v\|\psi w_c\|_0\bigl(1+\|\psi w_o\|_0+\|\psi w_c\|_0\bigr) +2\psi\Big|\int_{\T} w_o\cdot v\Big|\notag\\
&\leq C_v\psi\sqrt{{\delta^\ast} } D\mu\lambda^{\alpha-1}+2\psi\Big|\int  w_o\cdot v\Big|\label{E_est1}\\
&\leq C_{v,s}\psi\sqrt{{\delta^\ast} } D\mu\lambda^{\alpha-1}\label{E_est2}, 
\end{align}
where \eqref{E_est1} follows from Proposition \ref{P_w0wc_est} and the fact that $\lambda\geq(\mu D)^{1+\omega}$ (see \eqref{E_const_ineq}), while \eqref{E_est2} is a consequence of \eqref{E_osc_est_1} and Proposition \ref{P_doubling_ctheta}.

Now, tracing \eqref{E_double_W2} we obtain
\begin{align*}
 \Big|\int_{\T}\psi^2| w_o|^2-\psi^2\tr(R_\ell)\Big|&\leq\psi^2\sum_{1\leq|k|\leq2\lambda_0}\Big|\int_{\T} \tr U_{k}\Big|\\
&\leq C_e\psi^2{\delta^\ast} \mu D\lambda^{-1},
\end{align*}
the last inequality following from \eqref{E_osc_est_1} with $m=1$ and \eqref{E_ukr}.
\end{proof}

Then we conclude after observing that
\[
 \int_{\T}\tr(R_\ell)=e(1-\bar\delta)-\int_{\T}|v_\ell|^2
\]
and recalling \eqref{E_vvell_int}.

\subsection{Estimates on the Reynolds stress}
\label{S_Reynolds_est}

While estimating the Reynolds stress we refer to Proposition 8.1 in \cite{DS12H}. 
\begin{proposition}
\label{P_reyn_est_ctheta}
 For every $\alpha\in\bigl(0,\frac{\omega}{1+\omega}\bigr)$, there is a constant $C_{v,s}$ depending only on $\alpha$, $\omega$, $e$ and $\|v\|_{C^0}$ such that the following estimates hold
\begin{align}
  \|\mR_{1}-(1-\psi^2)\mR\|_0&\leq C_{v,s}\bigl[\psi D\ell+\psi\sqrt{{\delta^\ast} }\mu D\lambda^{2\alpha-1}+\psi\sqrt{{\delta^\ast} }\mu^{-1}\lambda^\alpha+|\psi'|\sqrt{{\delta^\ast} }\mu D\lambda^{\alpha-1}\bigr]\label{E_r1C0}\\
|||(\mR_1-(1-\psi^2)\mR)|||_1&\leq C_{v,s}\bigl[\psi\sqrt{{\delta^\ast} }\mu D\lambda^{2\alpha}+\psi\sqrt{{\delta^\ast} }\lambda D\ell+\psi\sqrt{{\delta^\ast} }\mu^{-1}\lambda^{\alpha+1}\notag\\
&+|\psi'|D\ell+|\psi'|\sqrt{{\delta^\ast} }\mu^{-1}\lambda^\alpha+|\psi'|\sqrt{{\delta^\ast} }\mu D\lambda^{\alpha}\notag\\
&+|\psi''|\sqrt{{\delta^\ast} }\mu D\lambda^{\alpha-1}\bigr]\label{E_r1C1}
\end{align}

\end{proposition}

\begin{proof}

Recalling the definition of $\mathring R_{1}$ made in \eqref{E_def_reyn} we set
\begin{equation}
\label{E_mR_dec}
 \mathring R_{1}- (1-\psi^2) \mathring R=\psi^2 \mR_1^1+\psi \mR_1^2+\psi^2 \mR_1^3+ \psi \mR_1^4 +\psi \mR_{1}^{5}+\psi \mR_1^6+\psi \mR_1^7+\psi'\mR^8_1
\end{equation}
where
\begin{align}
 \mR_1^1&=\mR-\mR_\ell\notag\\
 \mR_1^2&=w_1\otimes(v-v_\ell)+(v-v_\ell)\otimes w_1-\frac{2\langle (v-v_\ell),w_1\rangle}{3}\Id\notag\\
\mR_1^3&=\mathcal R\Big[\div\Big(w_o\otimes w_o+\mR_\ell-\frac{| w_o|^2}{2}\Id\Big)\Big]\notag\\
\mR_1^4&=\mathcal R(\partial_tw_c)\notag\\
\mR_1^{5}&=\mathcal R[\div((v_\ell+\psi w_1)\otimes w_c+w_c\otimes(v_\ell+\psi w_1)+w_c\otimes \psi w_c)]\notag\\
\mR_1^6&=\mathcal R\div( w_o\otimes v_\ell)\notag\\
\mR_1^7&=\mathcal R\bigl[\partial_t  w_o+v_\ell\cdot\nabla  w_o\bigr]\notag\\
\mR_1^8&=\mathcal R(w_1).\notag
\end{align}

Notice that $\mR_1^1,\,\mR_1^2,\,\mR_1^3,\,\mR_1^4,\,\mR_1^6$ and $\mR_1^7$ are defined as in the proof of Proposition 8.1 in \cite{DS12H}. Therefore, taking into account the multiplications by the cut-off function $\psi$ and the new estimates for the coefficients $a_k$ obtained in Proposition \ref{P_doubling_ctheta}, in the next table we collect without proof the related $C^0$ and $C^1$-estimates.

\vskip 0.3 cm
%\resizebox{15cm}{5cm}
%{
\begin{tabular}{|l|l|}
\hline
$C^0$ estimates & $C^1$ estimates\\
\hline
 $\|\psi^2\mR_1^1\|_0\leq \psi^2C D\ell$ & $|||\psi^2\mR_1^1|||_{1}\leq \psi^2CD+|\psi'|CD\ell$ \\
\hline
 $\|\psi\mR_1^2\|_0\leq \psi C_e\sqrt{{\delta^\ast} } D\ell$ & $|||\psi\mR_1^2|||_{1}\leq \psi C_v\sqrt{{\delta^\ast} }\lambda D\ell+|\psi'|C_e\sqrt{{\delta^\ast} } D\ell$ \\
\hline
 $\|\psi^2\mR_1^3\|_0\leq\psi^2 C_{e,s} {\delta^\ast} \mu D\lambda^{\alpha-1}$ & $|||\psi^2\mR_1^3|||_{1}\leq\psi^2C_{v,s}{\delta^\ast} \mu D\lambda^\alpha+|\psi'|C_{e,s}{\delta^\ast} \mu D\lambda^{\alpha-1}$\\
\hline
 $\|\psi\mR_1^4\|_0\leq \psi C_{v,s}\sqrt {{\delta^\ast} } \mu D\lambda^{\alpha-1}$ & $|||\psi\mR_1^4|||_{1}\leq \psi C_{v,s}\sqrt{{\delta^\ast} }\mu D\lambda^\alpha+|\psi'|C_{v,s}\sqrt{{\delta^\ast} }\mu D\lambda^{\alpha-1}$\\
\hline
  $\|\psi\mR_1^6\|_0\leq \psi C_{v,s}\sqrt{\delta^\ast} \mu D\lambda^{\alpha-1}$ & $|||\psi\mR_1^6|||_{1}\leq \psi C_{v,s}\sqrt{{\delta^\ast} }\mu D\lambda^\alpha+|\psi'|C_{v,s}\sqrt{{\delta^\ast} }\mu D\lambda^{\alpha-1}$\\ 
\hline
 $\|\psi\mR_1^7\|_0\leq\psi C_{v,s}\sqrt{ \delta^\ast}(\mu D\lambda^{\alpha-1}+\mu^{-1}\lambda^\alpha)$ & $|||\psi\mR_1^7|||_{1}\leq \psi C_{v,s}\sqrt{{\delta^\ast} }(\mu D\lambda^\alpha+\mu^{-1}\lambda^{\alpha+1})$\\
 & $\qquad\quad\quad+|\psi'|C_{v,s}\sqrt{{\delta^\ast} }(\mu D\lambda^{\alpha-1}+\mu^{-1}\lambda^\alpha)$\\ 
\hline
\end{tabular}
%}
\vskip 0.3 cm

Moreover, reasoning exactly as in Step 5 of the proof of Proposition 8.1 in \cite{DS12H} we get 
\begin{align}
 \|\psi\mR_{1}^5\|_0&\overset{\eqref{E_Rdiv}}{\leq} \|(v_\ell+\psi w_1)\otimes w_c+w_c\otimes(v_\ell+\psi w_1)+w_c\otimes \psi w_c\|_{\alpha}\notag\\
&\leq C_h\|w_c\|_{\alpha}(\|v_\ell\|_0+\psi\|w_o\|_{\alpha}+\psi\|w_c\|_{\alpha})\notag\\
&\overset{\text{Prop. } \ref{P_doubling_ctheta}}{\leq}C_{v,s}\sqrt{{\delta^\ast} }\mu D\lambda^{2\alpha-1}.\notag
\end{align}
and the same argument gives
\begin{align}
|||\psi \mR_{1}^5|||_{1}&\leq \psi C_{v,s}\sqrt{{\delta^\ast} }\mu D\lambda^{2\alpha}+|\psi'| C_{v,s}\sqrt{{\delta^\ast} }\mu D\lambda^{2\alpha-1}.\notag
\end{align}
Concerning $\psi'\mR_1^8$, recall that
\[
 w_c=\frac{1}{\lambda}\mathcal Qu_c, \quad u_c:=\sum_{|k|=\lambda_0}i\nabla a_k(x,t,\lambda t)\times\frac{k\times B_k}{|k|^2}e^{i\lambda x\cdot k},
\]
apply Propositions \ref{P_osc_est} and \ref{P_doubling_ctheta} and get
\begin{align*}
 \|\psi'\mR_1^8\|_0&\leq |\psi'| C_s\sum_{|k|=\lambda_0}\Big[\frac{1}{\lambda^{1-\alpha}}\|a_k\|_0+\frac{1}{\lambda^{m-\alpha}}[a_k]_m+\frac{1}{\lambda^m}[a_{k}]_{m+\alpha}\Big]\notag\\
&+|\psi'| C_s\sum_{|k|=\lambda_0}\Big[\frac{1}{\lambda^{2-\alpha}}\|a_k\|_1+\frac{1}{\lambda^{m+1-\alpha}}{[a_k]}_{m+1}+\frac{1}{\lambda^{m+1}}{[a_{k}]}_{m+1+\alpha}\Big].\notag\\
&\leq |\psi'|C_{v,s}\sqrt{{\delta^\ast} }\lambda^{\alpha-1}\Big[1+\lambda^{1-m}((\mu D)^m+\mu D\ell^{1-m})+\lambda^{1-m-\alpha}((\mu D)^{m+\alpha}+\mu D\ell^{1-m-\alpha})\Big]
\end{align*}
Choose now
\begin{equation*}
 m=\Big[\frac{1+\omega}{\omega}\Big]+1
\end{equation*}
and observe that
\[
 \frac{m}{m-1}\leq1+\omega.
\]
Then the last inequality in \eqref{E_const_ineq} becomes
\[
 \lambda^{1-m}\leq\min\bigl\{\ell^m,(\mu D)^{-m}\bigr\}
\]
and 
\begin{equation*}
  \|\psi' \mR_1^8\|_0\leq C_{v,s}|\psi'|\sqrt{{\delta^\ast} }\mu D\lambda^{\alpha-1}.
\end{equation*}

To estimate $|||\psi' \mR_1^8|||_1$ one can proceed exactly in the same way getting
\begin{equation}
\label{E:r12est}
|||\psi' \mR_1^8|||_1\leq |\psi''|C_{v,s}\sqrt{{\delta^\ast} }\mu D\lambda^{\alpha-1}+|\psi'|C_{v,s}\sqrt{{\delta^\ast} }\mu D\lambda^{\alpha}.
\end{equation}

Thus we conclude
\begin{align}
 \|\mR_{1}-(1-\psi^2)\mR\|_0&\leq C_{v,s}\bigl[\psi D\ell+\psi\sqrt{{\delta^\ast} }\mu D\lambda^{2\alpha-1}+\psi\sqrt{{\delta^\ast} }\mu^{-1}\lambda^\alpha+|\psi'|\sqrt{{\delta^\ast} }\mu D\lambda^{\alpha-1}\bigr]\notag\\
 |||(\mR_1-(1-\psi^2)\mR)|||_1&\leq C_{v,s}\bigl[\psi D+\psi\sqrt{{\delta^\ast} }\lambda D\ell+\psi\sqrt{{\delta^\ast} }\mu^{-1}\lambda^{\alpha+1}+\psi\sqrt{{\delta^\ast} }\mu D\lambda^{2\alpha}\notag\\
&+|\psi'|D\ell+|\psi'|\sqrt{{\delta^\ast} }\mu D\lambda^{\alpha}+|\psi'|\sqrt{{\delta^\ast} }\mu^{-1}\lambda^\alpha\notag\\
&+|\psi''|\sqrt{{\delta^\ast} }\mu D\lambda^{\alpha-1}\bigr]\notag\\
&\leq C_{v,s}\bigl[\psi\sqrt{{\delta^\ast} }\mu D\lambda^{2\alpha}+\psi\sqrt{{\delta^\ast} }\lambda D\ell+\psi\sqrt{{\delta^\ast} }\mu^{-1}\lambda^{\alpha+1}\notag\\
&+|\psi'|D\ell+|\psi'|\sqrt{{\delta^\ast} }\mu D\lambda^{\alpha}+|\psi'|\sqrt{{\delta^\ast} }\mu^{-1}\lambda^\alpha\notag\\
&+|\psi''|\sqrt{{\delta^\ast} }\mu D\lambda^{\alpha-1}\bigr],\notag
\end{align}
where in the last inequality we have used 
\[
 \sqrt{{\delta^\ast} }\mu D\overset{\eqref{E_const_ineq}}{\geq}D(\delta^\ast)^{-\frac{1}{2}}\geq D.
\]

\end{proof}

\subsection{Proof of Proposition \ref{P_pert_ctheta}}
\label{Ss:pert_ctheta}

After constructing $(v_1,p_1,\mR_1)$ as in Section \ref{S_construction_step}, we fix the perturbation parameters $\mu,\lambda,\ell$ and the H\"older exponent $\alpha$ in the estimates of Sections \ref{S_v_1_est}, \ref{S_energy_est} and \ref{S_Reynolds_est}. 

More precisely, in Propositions \ref{P_energy_est} and \ref{P_reyn_est_ctheta} choose --as in Section 9 of \cite{DS12H}-- 
\[
 \alpha=\frac{\omega}{2+\omega}=\frac{\eps}{4(\eps+1)}.
\]
and assume that
\[
 \ell=\frac{1}{L_v}\frac{\bar\delta}{D},
\]
where $L_v>\eta^{-1}$ is a sufficiently large constant depending only on $\|v\|_{C^0}$ and $e$ in such a way that \eqref{E:Dell} is satisfied. Then, impose
\begin{equation}
\label{E_mudlambda}
 \mu^2 D=\lambda,
\end{equation}
where actually, as observed in \cite{DS12H}, to be sure to satisfy \eqref{E_lm_int} one can carry on analogous computations imposing $\frac{\lambda}{2}\leq\mu^2 D\leq\lambda$. In particular,
 \begin{equation}
\label{E_mu-1lambda}
 \mu^{-1}\lambda^\alpha=\mu D\lambda^{\alpha-1}.
\end{equation}

Finally, choose $\lambda$ of the form 
\begin{equation}
\label{E_lambdadef}
 \lambda=\Lambda_v\Big(\frac{D\hat\delta}{\bar\delta^{2+\nu}}\Big)^{1+\eps}=\Lambda_v\Big(\frac{D\hat\delta}{\bar\delta^{2+\nu}}\Big)^{\frac{1+\omega}{1-\omega}},
\end{equation}
where $\Lambda_v\geq1$ depends only on $\|v\|_{C^0}$ and 
\begin{equation}
\label{E:nu}
 \nu=\frac{4}{9}.
\end{equation}
Observe that in \cite{DS12H} the authors set $\nu=0$.

Let us check that the choice of the parameters is consistent with \eqref{E_const_ineq}.
If $\Lambda_v\geq L_v^{1+\omega}$,
\begin{align}
 \lambda\ell^{1+\omega}&=\frac{\Lambda_v}{L_v^{1+\omega}}\Big(\frac{D\hat\delta}{\bar\delta^{2+\nu}}\Big)^{\frac{1+\omega}{1-\omega}}\Big(\frac{\bar\delta}{D}\Big)^{1+\omega}\geq\Big(D^\omega\frac{\hat\delta\bar\delta^{\frac{2}{2+\eps}}}{\bar\delta^{2+\nu}}\Big)^{\frac{1+\omega}{1-\omega}}\geq 1\notag
\end{align}
and since $\mu D=\lambda^{\frac{1}{2}}D^{\frac{1}{2}}$
\begin{align}
 \lambda^{-1}(\mu D)^{(1+\omega)}&=\lambda^{\frac{\omega-1}{2}}D^{\frac{1+\omega}{2}}=\Lambda_v^{\frac{\omega-1}{2}}D^{\frac{1+\omega}{2}}\Big(\frac{D\hat\delta}{\bar\delta^{2+\nu}}\Big)^{-\frac{1+\omega}{2}}=\Lambda_v^{\frac{\omega-1}{2}}\Big(\frac{\bar\delta^{2+\nu}}{\hat\delta}\Big)^{\frac{1+\omega}{2}}\leq1\notag
\end{align}
Moreover, by our choice of $\alpha$
\begin{equation}
\label{E_mudla-1}
 \mu D\lambda^{\alpha-1}\leq1.
\end{equation}
This follows from $\mu D\leq\lambda^{\frac{1}{\omega+1}}$ and the fact that $\frac{1}{\omega+1}\leq 1-\alpha$. Finally,
\begin{align*}
 \mu&=\sqrt{\frac{\lambda}{D}}=\Lambda_v^{\frac{1}{2}}D^{\frac{\eps}{2}}\Big(\frac{\hat\delta}{\bar\delta^{2+\nu}}\Big)^{\frac{1+\eps}{2}}\geq\bar\delta^{-\frac{2}{3}-\frac{\nu}{2}}\geq{\delta^\ast}^{-1}.
\end{align*}
%this is why you need to put $\hat\delta$ and not $\delta'$ at the numerator of $\lambda$!

Being now $\alpha$ fixed, the constants $C_{e,s}$ and $C_{v,s}$ which in general depend on $\alpha$ will be from now onwards denoted by $C_e$ and  $C_v$.

\underline{\it Preliminary estimates}

Before inserting the above parameters in the estimates of Sections \ref{S_v_1_est}, \ref{S_energy_est} and \ref{S_Reynolds_est}, it is convenient to estimate separately some of the terms which will often appear in the calculations.

By our choice of $\alpha$
\begin{equation}
 \label{E_alpha_form}
 \alpha-\frac{1}{2}=-\frac{\eps+2}{4(\eps+1)},\quad2\alpha-\frac{1}{2}=-\frac{1}{2(1+\eps)},\quad\alpha+\frac{1}{2}=\frac{3\eps+2}{4(\eps+1)},\quad2\alpha+\frac{1}{2}=\frac{2\eps+1}{2(1+\eps)}.
\end{equation}
Moreover, by definitions \eqref{E:md} and \eqref{E:deltaast}
\begin{equation}
\label{E_sqrtast}
 \sqrt{\frac{{\delta^\ast} }{\hat\delta}}\leq \sqrt{\frac{\delta''}{\hat\delta}}=\hat\delta^{-\frac{5}{18}}\leq\bar\delta^{-\frac{5}{27}}.
\end{equation}

Then we have the following estimates:
\begin{align}
 \sqrt{{\delta^\ast}}\mu D\lambda^{\alpha-1}&=\sqrt{{\delta^\ast} } D^{\frac{1}{2}}\lambda^{\alpha-\frac{1}{2}}=\sqrt{{\delta^\ast} } D^{\frac{1}{2}}\Lambda_v^{-\frac{\eps+2}{4(\eps+1)}}\Big(\frac{D\hat\delta}{\bar\delta^{2+\nu}}\Big)^{-\frac{\eps}{4}-\frac{1}{2}}\notag\\
&\leq\Lambda_v^{-\frac{\eps+2}{4(\eps+1)}}\sqrt{\frac{{\delta^\ast} }{\hat\delta}}\bar\delta^{\frac{11}{9}+\frac{11}{18}\eps}\hat\delta^{-\frac{\eps}{4}}\leq\Lambda_v^{-\frac{\eps+2}{4(\eps+1)}}\sqrt{\frac{{\delta^\ast} }{\hat\delta}}\bar\delta^{\frac{11}{9}+\frac{4}{9}\eps}\overset{\eqref{E_sqrtast}}{\leq}\Lambda_v^{-\frac{\eps+2}{4(\eps+1)}}\bar\delta^{\frac{28}{27}+\frac{4}{9}\eps}\label{E_alpha-1}\\
 \sqrt{{\delta}^\ast}\mu D\lambda^{2\alpha-1}&=\sqrt{{\delta^\ast} } D^{\frac{1}{2}}\lambda^{2\alpha-\frac{1}{2}}=\sqrt{{\delta^\ast} } D^{\frac{1}{2}}\Lambda_v^{-\frac{1}{2(1+\eps)}}\Big(\frac{D\hat\delta}{\bar\delta^{2+\nu}}\Big)^{-\frac{1}{2}}\notag\\
&=\Lambda_v^{-\frac{1}{2(1+\eps)}}\sqrt{\frac{{\delta^\ast} }{\hat\delta}}\bar\delta^{\frac{11}{9}}\overset{\eqref{E_sqrtast}}{\leq}\Lambda_v^{-\frac{1}{2(1+\eps)}}\bar\delta^{\frac{28}{27}} \label{E_2alpha-1}
\end{align}

\begin{align}
 \sqrt{{\delta}^\ast}\mu D\lambda^{2\alpha}&=\sqrt{{\delta^\ast} } D^{\frac{1}{2}}\lambda^{2\alpha+\frac{1}{2}}=\sqrt{{\delta^\ast} } D^{\frac{1}{2}}\Lambda_v^{\frac{2\eps+1}{2(\eps+1)}}\Big(\frac{D\hat\delta}{\bar\delta^{2+\nu}}\Big)^{\eps+\frac{1}{2}}\notag\\
&\leq\Lambda_v^{\frac{2\eps+1}{2(\eps+1)}}\sqrt{\frac{{\delta^\ast} }{\hat\delta}}\hat\delta^{\frac{3}{2}+\eps}\bar\delta^{\frac{4}{9}-\frac{4}{9}\eps}\Big(\frac{D}{\bar\delta^{2}}\Big)^{1+\eps}\overset{\eqref{E_sqrtast}}{\leq}\Lambda_v^{\frac{2\eps+1}{2(\eps+1)}}\hat\delta^{\frac{3}{2}+\eps}\bar\delta^{\frac{7}{27}-\frac{4}{9}\eps}\Big(\frac{D}{\bar\delta^2}\Big)^{1+\eps}\label{E_2alphadiff}
\end{align}

\begin{align}
 \sqrt{{\delta}^\ast}\mu D\lambda^{\alpha}&=\sqrt{{\delta^\ast} } D^{\frac{1}{2}}\lambda^{\alpha+\frac{1}{2}}=\sqrt{{\delta^\ast} } D^{\frac{1}{2}}\Lambda_v^{\frac{3\eps+2}{4(\eps+1)}}\Big(\frac{D\hat\delta}{\bar\delta^{2+\nu}}\Big)^{\frac{3}{4}\eps+\frac{1}{2}}\notag\\
&\leq\Lambda_v^{\frac{3\eps+2}{4(\eps+1)}}\sqrt{\frac{{\delta^\ast} }{\hat\delta}}\hat\delta^{\frac{3}{2}+\frac{3}{4}\eps}\bar\delta^{\frac{4}{9}+\frac{\eps}{6}}\Big(\frac{D}{\bar\delta^2}\Big)^{1+\eps}\leq \Lambda_v^{\frac{3\eps+2}{4(\eps+1)}}\hat\delta^{\frac{3}{2}+\frac{3}{4}\eps}\bar\delta^{\frac{7}{27}+\frac{\eps}{6}}\Big(\frac{D}{\bar\delta^2}\Big)^{1+\eps}\label{E_alphadiff}\\
 \sqrt{\delta^\ast}\lambda&=\sqrt{\frac{\delta^\ast}{\hat\delta}}\hat\delta^{\frac{3}{2}+\eps}\Big(\frac{D}{\bar\delta^2}\Big)^{1+\eps}\bar\delta^{-\frac{4}{9}-\frac{4}{9}\eps}\leq\hat\delta^{\frac{3}{2}+\eps}\Big(\frac{D}{\bar\delta^2}\Big)^{1+\eps}\bar\delta^{-\frac{17}{27}-\frac{4}{9}\eps}\label{E_deltasqrtlambda}
\end{align}

\underline{\it Proof of \eqref{E_e_dec}}
\vskip 0.1 cm

By Proposition \ref{P_energy_est}
\begin{align*}
\Big|e-\int_{\T}|v_1|^2-\Big[\psi^2\bar\delta e+(1-\psi^2)\Big(e-\int_{\T}|v|^2\Big)\Big]\Big|&\leq \psi^2 C_{v}D\ell+\psi C_v\sqrt{{\delta^\ast} }\mu D\lambda^{\alpha-1}\\
&\overset{\eqref{E_alpha-1}}{\leq}\psi^2 \frac{C_v}{L_v}\bar\delta+\psi\frac{C_v}{\Lambda_v^{\frac{1}{2(1+\eps)}}}\bar\delta^{\frac{28}{27}+\frac{4}{9}\eps}
\end{align*}
Hence, provided $L_v$ and consequently $\Lambda_v$ are chosen sufficiently big depending on $\|v\|_{C_0}$, we get \eqref{E_e_dec}.

\underline{\it Proof of \eqref{E_wC0} and \eqref{E_wC1}}. 
\vskip 0.1 cm
We use the estimates obtained in Proposition \ref{P_w0wc_est}.

Taking \eqref{E_w0Cr} for $r=0$ and \eqref{E_wcCr} for $r=\alpha$ we have
\begin{align}
 \|\psi w_1\|_0&\leq\psi C_e\bigl(\sqrt{{\delta^\ast} }+\sqrt{{\delta^\ast} }\mu D\lambda^{\alpha-1}\bigr)\notag\\
&\overset{\eqref{E_mudla-1}}{\leq}\psi M\sqrt{{\delta^\ast} },\label{E_w0proof}
\end{align}
provided the constant $M=M(e)$ is large enough.

Consider now \eqref{E_w0Cr} with $r=1$, \eqref{E_w0Crt} with $r=0$, \eqref{E_wcCr} with $r=\alpha+1$ and \eqref{E_wcCrt} with $r=\alpha$. Then
\begin{align}
\label{E_psiw1estfin}
 |||\psi w_1|||_1&\leq \psi C_e\sqrt{{\delta^\ast} }\lambda+|\psi'|C_v\sqrt{{\delta^\ast} }+\psi C_v\sqrt{{\delta^\ast} }\mu D\lambda^\alpha+|\psi'| C_v\sqrt{{\delta^\ast} }\mu D\lambda^{\alpha-1}\\
&\overset{\eqref{E_mudla-1}}{\leq} \psi C_v\sqrt{{\delta^\ast} }\lambda+|\psi'|C_v\sqrt{{\delta^\ast} }\notag\\
&\overset{\eqref{E_deltasqrtlambda}}{\leq} \psi \Lambda_v\hat\delta^{\frac{3}{2}+\eps}\Big(\frac{D}{\bar\delta^{2}}\Big)^{1+\eps}\bar\delta^{-\frac{17}{27}-\frac{4}{9}\eps}+|\psi'|C_v\sqrt{{\delta^\ast} }.\notag
\end{align}

In particular, since by \eqref{E_mudlambda} and \eqref{E_const_ineq} it is fairly easy to check that $D\leq \sqrt{{\delta^\ast} }\lambda$, 
\begin{equation*} 
 |||v_1|||_1\leq2\Lambda_v\hat\delta^{\frac{3}{2}+\eps}\Big(\frac{D}{\bar\delta^{2}}\Big)^{1+\eps}\bar\delta^{-\frac{17}{27}-\frac{4}{9}\eps}+|\psi'|C_v\sqrt{{\delta^\ast} }.
\end{equation*}

\underline{\it Proof of \eqref{E_pC0}}
Using \eqref{E:pc0} and \eqref{E_w0proof} we get
\begin{align}
 \label{E:pc0bis}
\|p_1-p\|_0&\leq \psi^2\frac{M^2}{4}{\delta^\ast} +\psi \frac{3}{4}M\bar\delta\sqrt{{\delta^\ast} }\leq\psi {M^2}{\delta^\ast} .
\end{align}

\underline{\it Proof of \eqref{E_r110}, \eqref{E_r120}, \eqref{E_r111} and \eqref{E_r121}.}
\vskip 0.1 cm
From \eqref{E_mR_dec}, define
\begin{align*}
\mathring R_{1,1}&:=\psi \mR_1^1+\mR_1^2+\psi \mR_1^3+ \mR_1^4 +\mR_{1}^{5}+\mR_1^6+\mR_1^7,\\
\mR_{1,2}&:=\mR^8_1.
\end{align*}

By \eqref{E_r1C0} and \eqref{E_mu-1lambda}

\begin{align*}
 \|\psi\mR_{1,1}\|_0&\leq \psi C_v D\ell+ \psi C_v\sqrt{{\delta^\ast} }\mu D\lambda^{2\alpha-1}\overset{\eqref{E_2alpha-1}}{\leq}\psi\frac{C_v}{L_v}\bar\delta+\psi\frac{C_v}{\Lambda_v^{\frac{1}{2(1+\eps)}}}\bar\delta^{\frac{28}{27}}\leq\psi\frac{\eta}{2}\bar\delta,
\end{align*}
provided $L_v$ and then $\Lambda_v$ are chosen large enough.

Moreover, 
\begin{equation*}
 \|\psi'\mR_{1,2}\|_0\leq |\psi'|C_v\sqrt{{\delta^\ast} }\mu D\lambda^{\alpha-1}\overset{\eqref{E_alpha-1}}{\leq}|\psi'|\frac{C_v}{\Lambda_v^{\frac{\eps+2}{4(\eps+1)}}}\bar\delta^{\frac{28}{27}+\frac{4}{9}\eps}.
\end{equation*}
By \eqref{E_r1C1}, \eqref{E_mu-1lambda}, \eqref{E_2alphadiff}, \eqref{E_2alpha-1} and \eqref{E_deltasqrtlambda}
\begin{align*}
 |||\psi\mR_{1,1}|||_1&{\leq}\psi \,C_v\bigl[\sqrt{{\delta^\ast} }\mu D\lambda^{2\alpha}+\sqrt{{\delta^\ast} }\lambda D\ell\bigr]+|\psi'|C_v\sqrt{{\delta^\ast} }\mu D\lambda^{2\alpha-1}\\
&{\leq}\psi \,C_v\Big[\Lambda_v^{\frac{2\eps+1}{2(\eps+1)}}\hat\delta^{\frac{3}{2}+\eps}\bar\delta^{\frac{7}{27}-\frac{4}{9}\eps}\Big(\frac{D}{\bar\delta^{2}}\Big)^{1+\eps}+\frac{\Lambda_v}{L_v}\hat\delta^{\frac{3}{2}+\eps}\bar\delta^{\frac{10}{27}-\frac{4}{9}\eps}\Big(\frac{D}{\bar\delta^{2}}\Big)^{1+\eps}\Big]+|\psi'|\frac{C_v}{\Lambda_v^{\frac{1}{2(1+\eps)}}}\bar\delta^{\frac{28}{27}}
\end{align*}
and by \eqref{E:r12est}, \eqref{E_alpha-1} and \eqref{E_alphadiff}
\begin{align*}
 |||\psi'\mR_{1,2}|||_1&\leq|\psi'|C_v\sqrt{{\delta^\ast} }\mu D\lambda^{\alpha}+|\psi''|C_v\sqrt{{\delta^\ast} }\mu D\lambda^{\alpha-1}\\
&{\leq}|\psi'| C_v\Lambda_v^{\frac{3\eps+2}{4(\eps+1)}}\hat\delta^{\frac{3}{2}+\frac{3}{4}\eps}\bar\delta^{\frac{7}{27}+\frac{\eps}{6}}\Big(\frac{D}{\bar\delta^2}\Big)^{1+\eps}+|\psi''|C_v\Lambda_v^{-\frac{\eps+2}{4(\eps+1)}}\bar\delta^{\frac{28}{27}+\frac{4}{9}\eps}.
\end{align*}

We thus end up with \eqref{E_r111} and \eqref{E_r121} again choosing suitable constants $L_v$ and $\Lambda_v$ and $A_1$ sufficiently big.

\begin{proof}
 [Proof of Corollary \ref{C_pert_ctheta}]
Let $\psi\in C^\infty([a_1,a_2];[0,1])$ be a function satisfying \eqref{E_psi'bounds}. In the following we will extensively use the relations
\begin{equation}
\label{E_eqdeltabarhat'}
 \bar\delta\leq\hat\delta^{\frac{3}{2}}=\delta'^{\frac{9}{4}}.
\end{equation}
To prove \eqref{E_r120diff} recall \eqref{E_r110} and observe that
\begin{align}
 \|\psi'\mR_{1,2}\|_0&\overset{\eqref{E_r120}}{\leq}|\psi'|\frac{\eta}{2\bar C}\bar\delta^{\frac{28}{27}+\frac{4}{9}\eps}\leq\frac{\eta}{2}\bar\delta\bigl[(\delta')^{-\frac{3\eps}{4}}\bar\delta^{\frac{1}{27}+\frac{4}{9}\eps}\bigr]\overset{\eqref{E_eqdeltabarhat'}}{\leq}\frac{\eta}{2}\bar\delta.\notag
\end{align}
As for \eqref{E_r121diff}, let us first consider \eqref{E_r111}. On one hand, since $\eps<\frac{7}{12}$,
\begin{align*}
 \psi A_1\hat\delta^{\frac{3}{2}+\eps}\Big(\frac{D}{\bar\delta^2}\Big)^{1+\eps}\bar\delta^{\frac{7}{27}-\frac{4}{9}\eps}\leq\psi A_1\hat\delta^{\frac{3}{2}}\Big(\frac{D}{\bar\delta^2}\Big)^{1+\eps}.
\end{align*}
On the other hand, 
\[
 |\psi'|\frac{1}{\bar C}\bar\delta^{\frac{28}{27}}\overset{\eqref{E_psi'bounds}}{\leq}\chi_{\{\psi>0\}}\bar\delta^{\frac{19}{27}}\leq\chi_{\{\psi>0\}},\qquad \hat\delta^{\frac{3}{2}}\Big(\frac{D}{\bar\delta^2}\Big)^{1+\eps}\geq1,
\]
so that
\[
 |||\psi\mR_{1,1}|||_1\leq\chi_{\{\psi>0\}}A_1\hat\delta^{\frac{3}{2}}\Big(\frac{D}{\bar\delta^2}\Big)^{1+\eps}.
\]
To show that also $ |||\psi'\mR_{1,2}|||_1\leq\chi_{\{\psi>0\}}A_1\hat\delta^{\frac{3}{2}}\Big(\frac{D}{\bar\delta^2}\Big)^{1+\eps}$, it is sufficient to notice that
\begin{align*}
 |\psi'|\frac{1}{\bar C}\hat\delta^{\frac{3}{4}\eps}\leq1,\qquad|\psi''|\frac{1}{\bar C^2}\bar\delta^{\frac{28}{27}+\frac{4}{9}\eps}\leq1,\qquad\hat\delta^{\frac{3}{2}}\Big(\frac{D}{\bar\delta^2}\Big)^{1+\eps}\geq1.
\end{align*}

Finally, to prove \eqref{E_wC1eq}, we choose $\nu=0$ in \eqref{E_lambdadef} as in \cite{DS12H}, i.e. we set
\begin{equation}
\label{E_lambdadef2}
 \lambda=\Lambda_v\Big(\frac{D\hat\delta}{\bar\delta^2}\Big)^{1+\eps}
\end{equation}
as in \cite{DS12H}. Since the $\lambda$ defined in \eqref{E_lambdadef2} is smaller than the $\lambda$ defined in \eqref{E_lambdadef}, then the $C^1$ estimates for the Reynolds stress proved before hold as well. In order to prove the energy estimate and the $C^0$ estimates, notice that if $\delta^\ast=\hat\delta$
 \begin{align*}
  \sqrt{\hat\delta}\mu D\lambda^{\alpha-1}&=\sqrt{{\hat\delta} } D^{\frac{1}{2}}\Lambda_v^{-\frac{\eps+2}{4(\eps+1)}}\Big(\frac{D\hat\delta}{\bar\delta^{2}}\Big)^{-\frac{\eps}{4}-\frac{1}{2}}\leq\Lambda_v^{-\frac{\eps+2}{4(\eps+1)}}\bar\delta\hat\delta^{\frac{\eps}{2}},\\
\sqrt{\hat\delta}\mu D\lambda^{2\alpha-1}&=\sqrt{{\hat\delta} } D^{\frac{1}{2}}\Lambda_v^{-\frac{1}{2(\eps+1)}}\Big(\frac{D\hat\delta}{\bar\delta^{2}}\Big)^{-\frac{1}{2}}\leq\Lambda_v^{-\frac{1}{2(\eps+1)}}\bar\delta
 \end{align*}
and proceed as in the proof of Proposition \ref{P_pert_ctheta}. 
As for the improvement \eqref{E_wC1eq} in the $C^1$ norm of $v_1$,
\begin{align*}
 |||\psi w_1|||_1&\leq\psi C_v\sqrt{{\hat\delta} }\lambda\overset{\eqref{E_lambdadef2}}{\leq} \psi C_v\Lambda_v\hat\delta^{\frac{3}{2}+\eps}\Big(\frac{D}{\bar\delta^2}\Big)^{1+\eps}.
\end{align*}

\end{proof}

\section{Admissible initial data}
\label{S_subsolution}
\begin{proposition}
\label{P_subsolution}
 Let $e\in C^\infty([0,1])$, $e>0$. Then, for all $\theta<\frac{1}{10}$, $0<\beta\leq\frac{1}{2}$ there exist $v_0\in C^\theta(\T;\R^3)$ and $(v,p,{\mR})\in C^0(\T\times[0,1];\R^3\times\R\times\SS{0})\cap C^1(\T\times(0,1];\R^3\times\R\times\SS{0})$ which solve the Euler-Reynolds system \eqref{E:ersystem} on $\T\times(0,1)$ with
\begin{align*}
& \int_{\T}|v_0|^2=e(0),\quad{\mR}(0)\equiv0\\
&v(0)=v_0, \quad v(t)\in C^\theta(\T)\quad\forall\,t\in[0,1].
\end{align*}
Moreover, there exist \[
                      0<\eps=\eps(\theta)<1\text{ with }\underset{\theta\to\frac{1}{10}}{\lim}\,\eps(\theta)=0 
                      \]
and sequences
\begin{align}
&\delta_n=a^{-b^n},\quad a\geq\frac{3}{2},\,b=\frac{3}{2},\text{ with $a=a(\eps)$},\notag\\
&t_0:=1,\quad t_n\in (0,1)\text{ s.t. $t_n<t_{n-1}$ and ${t_{n-1}-t_n}=\frac{2}{\bar C}\delta_{n-1}^{\frac{3\eps}{4}}$, $\underset{n=1}{\overset{\infty}{\sum}}t_{n-1}-t_n=1$}\notag\\
&\varphi_n\in C^{\infty}([0,1];[0,1])\text{ s.t. }\varphi_n(t)=\left\{\begin{aligned}
& 1 && &\text{if $t\in[0,t_n]$}\\
&0 && &\text{if $t\in[t_{n-1},1]$}
\end{aligned}\right.\notag
\end{align}
such that
 \begin{align}
&\Big|\Big(e-\int|v|^2\Big)-\bigl[\varphi_n^2\delta_n+(1-\varphi_n^2)\delta_{n-1}\bigr]e\Big|\chi_{[t_n,t_{n-1}]}\leq\frac{\beta}{2}\bigl[\varphi_n^2\delta_n+(1-\varphi_n^2)\delta_{n-1}\bigr]e,\label{E_e_est_subs}\\
&\|\mathring R-(1-\varphi_n^2)\mR'_{n}\|_0\chi_{[t_n,t_{n-1}]}\leq\chi_{\{\varphi_n>0\}}\eta\delta_{n},\quad\|\mR'_{n}\|_0\leq\eta\delta_{n-1},\label{E_R_est_subs}
\end{align}
for some $\mR'_n\in C^1(\T\times[0,1];\SS{0})$, with $\delta_0=1$ and $\eta$ depending on $e$ as in Proposition \ref{P_pert_ctheta}.
\end{proposition}

\begin{definition}
\label{D_admissible_subs}
 A triple $( v,p, \mR)$ as in Proposition \ref{P_subsolution} will be called an \emph{admissible $C^\theta$-subsolution} with admissible initial datum $v_0$ for the energy $e$. 
\end{definition}

\begin{remark}
 In the proofs of Proposition \ref{P_subsolution} and Theorem \ref{T:main} we will repeatedly use the following fact. Let $\psi,\,\varphi\in C^1([0,1];[0,1])$ such that
\begin{align}
 \supp\,\psi\subset\{\varphi=1\}\label{12}.
\end{align}
Then,
\begin{equation}
 \label{12c}
(\psi+\psi')(1-\varphi)=0,\quad(1-\psi^2)(1-\varphi^2)=(1-\varphi^2),\quad\psi\varphi=\psi,
\end{equation}
\end{remark}

\begin{proof}[Proof of Proposition \ref{P_subsolution}]

 {\it \underline{Step 0.} } 
W.l.o.g., we assume that $\exists\,h>0$ such that $e$ is positive and smooth on the larger time interval $[-h,1+h]$.
Define then $(v_0,p_0,\mR_0)$ to be identically $0$ on $\T\times[-h,1+h]$ and observe that it satisfies the assumptions of Proposition \ref{P_pert_ctheta} on $[a_1,a_2]=[-h,1+h]$ with $\varphi\equiv1$ and $\hat\delta=1$. Fix $\eps>0$ arbitrarily small and set $\delta_0:=\hat\delta=1$, $D_0:=D\equiv 1$, $\ell_0=0$ and $\varphi_0:=\varphi\equiv 1$.

Starting from $(v_0,p_0,\mR_0)$ we construct iteratively --using Proposition \ref{P_pert_ctheta}-- a sequence of solutions $(v_n,p_n,\mR_n)$ of the Euler-Reynolds system which will converge to an admissible $C^\theta$ subsolution $(\bar v, \bar p,\bar{\mR})$. To this aim, we first set the following parameters: 
\begin{align}
 &\delta_n=a^{-b^n},\quad a\geq\frac{3}{2},\,b=\frac{3}{2},\text{ with $a=a(\eps)$ to be fixed later}\notag\\
&\bar C=4\sum_n \delta_{n-1}^{\frac{3}{4}\eps}=4\sum_n a^{-\frac{3\eps}{4}b^{n-1}}\notag\\
&\tilde\ell_n=\frac{2}{\bar C}\delta_{n-1}^{\frac{3\eps}{4}},\notag\\
&t_0:=1,\quad t_n\in (0,1)\text{ s.t. $t_n<t_{n-1}$ and ${t_{n-1}-t_n}=\frac{4}{\bar C}\delta_{n-1}^{\frac{3\eps}{4}}=2\tilde\ell_n$,}\notag\\
&\bar\ell_n>0\text{ s.t. ${\sum}_n\bar\ell_n<h$}\label{E_ln}
\end{align}
and a countable family of cut-off functions
\begin{align}
&\varphi_n\in C^{\infty}([-h,1+h];[0,1])\text{ s.t. }\varphi_n(t)=\left\{\begin{aligned}
& 1 && &\text{if $t\in[-h,t_n]$}\\
&0 && &\text{if $t\in[t_{n-1}-\tilde\ell_n,1+h]=[t_n+\tilde\ell_n,1+h]$}
\end{aligned}\right.\label{E_varphi_n}\\
&\text{and }|\varphi'_n|\leq\frac{4}{t_{n-1}-t_n}=\bar C\delta_{n-1}^{-\frac{3\eps}{4}},\quad |\varphi''_n|\leq\frac{4\max|\varphi'_n|}{t_{n-1}-t_n}\leq\bar C^2\delta_{n-1}^{-\frac{3\eps}{2}}.\label{E_varphi'_nbounds}
\end{align}
Notice that, if $a$ is sufficiently big, $\delta_n\leq\min\{\frac{1}{2}\delta_{n-1},\,\delta_{n-1}^{\frac{3}{2}}\}$ and, by the choice of $\bar C$, $\underset{n=1}{\overset{\infty}{\sum}}t_{n-1}-t_n=1$. 

Moreover, the functions $\psi=\varphi_n$, $\varphi=\varphi_{n-1}$ satisfy \eqref{12}  for all $n\in\N$. In particular,
\begin{equation}
\label{E_psivarphi1}
 \underset{i=k}{\overset{\bar k}{\prod}}(1-\varphi_{i}^2)=1-\varphi_{k}^2,\quad\underset{i=k}{\overset{\bar k}{\prod}}\varphi_i=\varphi_{\bar k},\quad (1-\varphi_i^2)\varphi_j=0 \text{ if $i<j$.}
\end{equation}

{\it \underline{Step 1.} Iterative perturbation step}  

Let us assume that, after $n$ steps of the iterative procedure there exist $\{\ell_i\}_{i=1}^n$, $0<\ell_i<\min\{\bar\ell_i,\tilde\ell_i\}$ and $\{(v_i,p_i,\mR_i)\}_{i=1}^n$ solutions of \eqref{E:ersystem} on $[-h+\sum_{i=1}^n\ell_i,1+h]$ satisfying
\begin{align}
 &v_n=v_{n-1}+\varphi_n w_n\label{E:vnsubs}\\
&p_n=p_{n-1}+\varphi_n p_{no}\label{E:pnsubs}\\
&\mR_n=\varphi_n\mR_{n,1}+\varphi_n'\mR_{n,2}+(1-\varphi_n^2)\mR_{n-1}\label{E:Rnsubs}
\end{align}
and the following estimates
\begin{align}
 &\Big|e-\int|v_n|^2-\Big[\varphi_n^2\delta_n +(1-\varphi_n^2)\Big(e-\int|v_{n-1}|^2\Big)\Big]\Big|\leq\varphi_n\frac{\beta}{4}\delta_ne\label{E:eineq_n}\\
&\|v_n-v_{n-1}\|_0\leq\varphi_nM\sqrt{\delta_{n-1}}\label{E:vineq_n}\\
&\|\mR_n-(1-\varphi_n^2)\mR_{n-1}\|_0\leq\chi_{\{\varphi_n>0\}}\eta\delta_n\label{E:Rineq_n}\\
&\|p_n-p_{n-1}\|_0\leq\varphi_n M^2\delta_{n-1}\label{E:pineq_n}\\
&D_n\leq A_{n-1}\delta_{n-1}^{\frac{3}{2}}\Big(\frac{D_{n-1}}{\delta_n^2}\Big)^{1+\eps},\label{E:Dineq_n}
\end{align}
where 
\[
 D_i:=\max\bigl\{1,\|v_i\|_{C^1},\|\mR_i\|_{C^1}\bigr\}
\]
and $A_i=A_i(e,\eps,\|v_i\|_{C^0})$.

In particular, on $\bigl[-h+\underset{i=1}{\overset{n}{\sum}}\ell_i,t_n\bigr]\subset\{\varphi_n=1\}$ 
\begin{align*}
 &\Big|e-\int|v_n|^2-\delta_ne\Big|\leq\frac{\beta}{4}\delta_ne\\
&\|\mR_n\|_0\leq\eta\delta_n.
\end{align*}

Therefore $(v_n,p_n,\mR_n)$ satisfies the assumptions of Proposition \ref{P_pert_ctheta} on $[a_1,a_2]=[-h+\underset{i=1}{\overset{n}{\sum}}\ell_i,t_n]$ with $\varphi\equiv 1$ (or equivalently $\delta^\ast=\hat\delta$) and $\hat\delta=\delta_n$. 
Then, since 
\begin{equation}
\label{E:supp_phin+1}
 \supp\varphi_{n+1}\subset[-h+\underset{i=1}{\overset{n}{\sum}}\ell_i,t_n-\tilde\ell_{n+1}],
\end{equation}
we can apply Corollary \ref{C_pert_ctheta} with $\psi=\varphi_{n+1}$, $\ell=\ell_{n+1}<\min\{\bar\ell_{n+1},\tilde\ell_{n+1}\}$, $\bar\delta=\delta_{n+1}$. 

We thus get \eqref{E:vnsubs}-\eqref{E:Dineq_n} on the time interval $[-h+\underset{i=1}{\overset{n+1}{\sum}}\ell_i,t_n-\ell_{n+1}]$, with $n$ replaced by $n+1$ . 

Now observe that, since $\varphi_{n+1}=0$ on $[t_n-\ell_{n+1},1+h]$, we can extend trivially $v_{n+1}$ to be equal to $v_n$, $p_{n+1}$ to $p_n$ and $\mR_{n+1}$ to $\mR_n$ on $[t_n-\ell_{n+1},1+h]$ and moreover also \eqref{E:vnsubs}-\eqref{E:Dineq_n} extend to the larger time interval $[-h+\underset{i=1}{\overset{n+1}{\sum}}\ell_i,1+h]$.

{\it \underline{Step 2:} Convergence to a subsolution}

By $\eqref{E_ln}$, the functions $(v_n,p_n,\mR_n)$ are well defined on $\T\times[0,1]$.

From \eqref{E:vineq_n}, \eqref{E:Rineq_n} and \eqref{E:pineq_n} we immediately get that
\[
 \exists\,(v,p,\mR)\in C^0(\T\times[0,1];\R^3\times\R\times\SS{0})\text{ s.t. }\underset{n\to+\infty}{\lim}\|v_n-v\|_{C^0}+\|p_n-p\|_{C^0}+\|\mR_n-\mR\|_{C^0}=0.
\]

In particular, there exists a constant $A=A(\eps,e,\bar C)$ such that
\[
 A_n\leq A, \quad\forall\,n\in\N,
\]
and from now onwards we will substitute it to $A_n$ into \eqref{E:Dineq_n}.

Moreover, since $\|\mR_n(0)\|_0\leq\eta\delta_n$ and $\Big|e(0)-\int|v_n|^2(0)\Big|\leq\Big(1+\frac{\beta}{4}\Big)\delta_ne(0)$ for all $n\in\N$,
\[
 \mR(0)\equiv0,\qquad\int_{\T}|v|^2(0)=e(0).
\]

From the structure equations \eqref{E:vnsubs}, \eqref{E:pnsubs} and \eqref{E:Rnsubs} and \eqref{E:supp_phin+1} it is also fairly easy to see that, $\forall\, n\in\N$
\begin{align}
 &v\chi_{[t_n,t_{n-1}]}=v_n\chi_{[t_n,t_{n-1}]}=v_0+\varphi_nw_n+\underset{i=1}{\overset{n-1}{\sum}}w_i\label{E:vvn}\\
 &p\chi_{[t_n,t_{n-1}]}=p_n\chi_{[t_n,t_{n-1}]}=p_0+\varphi_np_{no}+\underset{i=1}{\overset{n-1}{\sum}}p_{io}\label{E:ppn}\\
 &\mR\chi_{[t_n,t_{n-1}]}=\mR_n\chi_{[t_n,t_{n-1}]}=\varphi_n\mR_{n,1}+\varphi_n'\mR_{n,2}+(1-\varphi_{n}^2)\mR_{n-1,1}\label{E:RRn}
\end{align}

In particular $(v,p,\mR)$ satisfies \eqref{E:ersystem} on $\T\times(0,1)$.

Moreover, using \eqref{E:vvn} and the structural assumptions \eqref{E_psivarphi1} on the cut-off functions $\varphi_n$ we obtain 
\begin{align}
\Big(e-\int|v|^2\Big)\chi_{[t_n,t_{n-1}]}&=\varphi_n^2\delta_ne+(1-\varphi_n^2)\Big(e-\int|v_{n-1}|^2\Big)+\varphi_n\tilde e_n\notag\\
&=\varphi_n^2\delta_ne+(1-\varphi_n^2)\varphi_{n-1}^2\delta_{n-1}e+(1-\varphi_n^2)(1-\varphi_{n-1}^2)\delta_{n-2}e\notag\\
&+(1-\varphi_n^2)\varphi_{n-1}\tilde e_{n-1}+\varphi_{n}\tilde e_n\notag\\
&=\varphi_n^2\delta_ne+(1-\varphi_n^2)\delta_{n-1}e+(1-\varphi_n^2)\tilde e_{n-1}+\varphi_n\tilde e_n\label{E_e_dec_n_f}
\end{align}
where 
\[
 \varphi_i\tilde e_i:=e-\int|v_i|^2-\Big[\varphi_i^2\delta_i +(1-\varphi_i^2)\Big(e-\int|v_{i-1}|^2\Big)\Big],\qquad|\tilde e_i|\leq\frac{\beta}{4}\delta_ie.
\]
By Lemma \ref{L:bound} we have then \eqref{E_e_est_subs}.

{\it \underline{Step 3:} Convergence in the H\"older norm.} 

To show that the sequence $\{v_n(t)\}$ converges to $v(t)$ in $C^\theta(\T;\R^3)$ with exponent $\theta<\frac{1}{10}$ we use the same argument of \cite{DS12H}, thanks to the fact that the rate of decay (resp. growth) of the $C^0$ (resp. $C^1$) norms between two successive approximations are double exponentials with suitable exponents. More precisely, our claim is that $\forall\,n\in\N$
\begin{equation}
\label{E:dnc}
 D_n\leq a^{cb^{n}}
\end{equation}
with $c=\frac{3+8\eps}{1-2\eps}$, for some $a=a(\eps,e)$, provided $\eps<\frac{1}{4}$.
Notice that \eqref{E:dnc} holds for $n=0$ and assume that it holds also $\forall\,i\leq n-1$. Then
\begin{align}
 D_n&\overset{\eqref{E:Dineq_n}}{\leq} A\delta_{n-1}^{\frac{3}{2}}\Big(\frac{D_{n-1}}{\delta_n^2}\Big)^{1+\eps}\notag\\
&= A\delta_n^{-1-2\eps}D_{n-1}^{1+\eps}\notag\\
&\leq A a^{b^n\bigl(1+2\eps\bigr)+cb^{n-1}(1+\eps)}.\notag
\end{align}
Hence, it is sufficient to check as in \cite{DS12H} that $c=\frac{3+8\eps}{1-2\eps}$ satisfies
\begin{equation}
\label{E:cb}
 cb-\bigl[b(1+2\eps)+c(1+\eps)\bigr]\geq\eps,
\end{equation}
provided $\eps<\frac{1}{4}$.
Indeed, if \eqref{E:cb} holds,
\[
 D_n\leq A a^{cb^n} a^{-b^{n-1}\eps}
\]
and \eqref{E:dnc} follows provided $a\geq A^{\frac{1}{\eps}}$.

Finally estimate by interpolation
\begin{align}
 \|v_n-v_{n-1}\|_{C^\theta}&\leq\|v_n-v_{n-1}\|_{C^0}^{1-\theta}\|v_n-v_{n-1}\|_{C^1}^\theta\notag\\
&\leq a^{(-\frac{1}{2}(1-\theta)+cb\theta)b^{n-1}},\notag
\end{align}
which tends to zero as $n\to+\infty$ provided $(-\frac{1}{2}(1-\theta)+cb\theta)<0$. Hence, by the choice of $b$ and $c$, if 
\[
 \theta<\frac{1}{1+3c}=\frac{1-2\eps}{10+22\eps},
\]
which tends to $\frac{1}{10}$ as $\eps$ tends to $0$.
\end{proof}

\begin{remark}
\label{R_inf_subsol}
 The fact that the admissible initial data for a given energy $e$ are infinitely many can be verified in different ways. Even though we do not pursue this issue in detail here, since there are essentially no new ideas and the calculations would be very similar, one could check as in \cite{Cho} that the $H^{-1}$ norm of the admissible subsolution and of the starting vector field of the iteration of Proposition \ref{P_subsolution} can be made arbitrarily small. Thus, choosing a starting vector field different from the trivial one (e.g. the vector field of the iteration at step $j$) and a sufficiently big $a\geq\frac{3}{2}$ in the Definition of the $\{delta_n\}$ one can find a different admissible initial datum.
\end{remark}

\section{Proof of Theorem \ref{T:main}}
\label{S_solution}

Let $v_0$ be an admissible initial datum for the energy $e$. From Step 0 to Step 3 we prove the existence of an H\"older solution of \eqref{E:euler} with exponent $\theta<\frac{1}{16}$ and total kinetic energy $e$. In Step 4 we give a short proof of the fact that actually there are infinitely many solutions satisfying the same requirements, and we also explain why the admissible initial data for \eqref{E:euler} w.r.t. any preassigned total kinetic energy are infinitely many, too.  

{\it \underline{Step 0.} } 
Let $(v,p,\mR)\in C^0(\T\times[0,1];\R^3\times\R\times\SS{0})\cap C^1(\T\times (0,1];\R^3\times\R\times\SS{0})$ be an admissible subsolution as in Proposition \ref{P_subsolution} with initial datum $v_0$ for some $\theta<\frac{1}{10}$ and $\beta=\frac{1}{8}$. W.l.o.g., we assume that $\exists\,h>0$ such that $e$ is positive and smooth on the larger time interval $[0,1+h]$ and $(v,p,\mR)$ is also defined and satisfies the same properties on $[0,1+h]$.

Set $(\bar v_0,\bar p_0,\bar\mR_0):=(v,p,\mR)$, $0<\eps:=\eps(\theta)<1$, $t_0=1$ and introduce, analogously to what we did in the proof of Proposition \ref{P_subsolution}, a countable family of cut-off functions
\begin{align}
&\psi_n\in C^{\infty}([0,1+h];[0,1])\text{ s.t. }\psi_n(t)=\left\{\begin{aligned}
& 1 && &\text{if $t\in[t_{n-1},1+h]$}\\
&0 && &\text{if $t\in[0,t_{n-1}-\tilde\ell_n]=[0,t_n+\tilde\ell_n]$}
\end{aligned}\right.\label{E_psi_n}\\
&\text{and }|\psi'_n|\leq\frac{4}{t_{n-1}-t_n}=\bar C\delta_{n-1}^{-\frac{3\eps}{4}},\quad |\psi''_n|\leq\frac{4\max|\psi'_n|}{t_{n-1}-t_n}\leq\bar C^2\delta_{n-1}^{-\frac{3\eps}{2}},\label{E_psi'_nbounds}
\end{align}
where
\[
 \tilde\ell_n=\frac{2}{\bar C}\delta_{n-1}^{\frac{3\eps}{4}}=\frac{1}{2}|t_n-t_{n-1}|, 
\]
and the following parameters
\begin{align}
&\bar\delta_n=\delta_{n+1},\notag\\
&\md_n=\varphi_n^2\delta_n+(1-\varphi_n^2)\delta_{n-1},\notag\\
&\delta^\ast_n=\psi_n^2\bar\delta_n+(1-\psi_n^2)\md_n\notag\\
&\bar\ell_n>0\text{ s.t. ${\sum}_n\bar\ell_n<h$.}\label{E_lnsol}
\end{align}

Notice that, in contrast with what happens for the cut-off functions $\varphi_n$, the set $\{\psi_{n}=1\}$ contains an $\tilde\ell_{n-1}$-neighborhood of the support of $\psi_{n-1}$ and then, by \eqref{12} with $\psi=\psi_{n-1}$, $\varphi=\psi_n$
\begin{equation}
 \label{E_psivarphi2}
 \underset{i=k}{\overset{\bar k}{\prod}}(1-\psi_{i}^2)=1-\psi_{\bar k}^2,\quad\underset{i=k}{\overset{\bar k}{\prod}}\psi_i=\psi_{k},\quad (1-\psi_i^2)\psi_j=0 \text{ if $i>j$.}
\end{equation}

Assume for the moment that 
\begin{equation}
\label{E_delta0}
 \delta_0=a^{-1}\leq \frac{1}{4}.
\end{equation}
At the end of the proof it will be easy to notice that this is not a restrictive assumption and the case $\delta_0=1$ can be handled with no additional difficulties. If \eqref{E_delta0} holds, we have in particular that $\delta_1=\delta_0^{\frac{3}{2}}\leq\frac{1}{2}\delta_0$.
In particular, by \eqref{E_e_est_subs} and \eqref{E_R_est_subs}, $(\bar v_0,\bar p_0,\bar {\mR}_0)$ satisfies the assumptions of Proposition \ref{P_pert_ctheta} on $[a_1,a_2]=[t_1,1+h]$ with $\varphi=\varphi_1$, $\hat\delta=\delta_1$, $\phi\equiv 1$, $\delta'=\delta_0$ and $\beta=\bar\beta:=\frac{1}{8}$. Set then $\bar D_0:=\max\{{1,\|\bar v_0\|_{C^1([t_1,1+h])},\|\bar p_0\|_{C^1([t_1,1+h])},\|\bar {\mR}_0\|_{C^1([t_1,1+h])}\}}$, $\ell_0=0$ and $0\equiv\psi_0\in C^\infty([0,1+h];[0,1])$.

Starting from $(\bar v_0,\bar p_0,\bar{\mR}_0)$ we construct iteratively --using Proposition \ref{P_pert_ctheta}-- a sequence of solutions $(\bar v_n,\bar p_n,\bar{\mR}_n)$ of the Euler-Reynolds system which will converge to a solution $(\bar v, \bar p,\bar{\mR})$ of \eqref{E:euler} with initial datum $\bar v_0(0)$.

{\it \underline{Step 2.} Iterative perturbation step}  

Let us assume that, after $n$ steps of the iterative procedure there exist $\{\ell_i\}_{i=1}^n$, $0<\ell_i<\min\{\bar\ell_i,\tilde\ell_i\}$ and $\{(\bar v_i,\bar p_i,\bar{\mR}_i)\}_{i=1}^n$ solutions of \eqref{E:ersystem} on $[0,1+h-\sum_{i=1}^n\ell_i]$ satisfying
\begin{align}
 &\bar v_i=\bar v_{i-1}+\psi_i \bar w_i\label{E:vnsol}\\
&\bar p_i=\bar p_{i-1}+\psi_i \bar p_{io}\label{E:pnsol}\\
&\bar \mR_i=\psi_i\bar \mR_{i,1}+\psi_i'\bar\mR_{i,2}+(1-\psi_i^2)\bar\mR_{i-1}\label{E:Rnsol}
\end{align}
and the following estimates
\begin{align}
 &\Big|e-\int|\bar v_i|^2-\Big[\psi_i^2\bar\delta_i +(1-\psi_i^2)\Big(e-\int|\bar v_{i-1}|^2\Big)\Big]\Big|\leq\psi_i\Big(\chi_{[t_i, t_{i-1}]}\frac{\bar\beta}{2}+\chi_{[t_{i-1}, 1+h-\underset{j=1}{\overset{i}{\sum}}\ell_j]}{\bar\beta}\Big)\bar\delta_ie\label{E:eineq_nsol}\\
&\|\bar v_i-\bar v_{i-1}\|_0\leq\psi_iM\sqrt{\md_{i}}\label{E:vineq_nsol}\\
&\|\bar\mR_i-(1-\psi_i^2)\bar\mR_{i-1}\|_0\leq\chi_{\{\psi_i>0\}}\eta\bar\delta_i\label{E:Rineq_nsol}\\
&\|\bar p_i-\bar p_{i-1}\|_0\leq\psi_i M^2\md_{i}\label{E:pineq_nsol}\\
&\bar D_i\leq A_{i-1}\bar\delta_{i-1}^{\frac{3}{2}}\Big(\frac{\bar D_{i-1}}{\bar\delta_i^2}\Big)^{1+\eps}\bar\delta_{i}^{-\frac{17}{27}-\frac{4}{9}\eps},\label{E:Dineq_nsol}
\end{align}
where 
\[
 \bar D_i:=\max\bigl\{1,\|\bar v_i\|_{C^1},\|\bar\mR_i\|_{C^1}\bigr\}
\]
and $A_i=A_i(e,\eps,\|\bar v_i\|_{C^0})$.

\underline{\emph{Claim}:} Our goal is now to show that $(\bar v_n, \bar p_n, \bar\mR_n)$ satisfies the assumptions \eqref{E_pert_ctheta_hp_ev} and \eqref{E_pert_ctheta_r0} of Proposition \ref{P_pert_ctheta} on $[a_1,a_2]=[t_{n}, 1+h-\underset{i=1}{\overset{n}{\sum}}\ell_i]$ with $\delta^\ast=\delta^\ast_n$, $\hat\delta=\bar\delta_n=\delta_{n+1}$, $\md=\md_n$, $\beta=2\bar\beta$ and on $[a_1,a_2]=[t_{n+1},t_n]$ with $\delta^\ast=\md_{n+1}$, $\hat\delta=\bar\delta_n=\delta_{n+1}$, $\md=\delta_n$, $\beta=\bar\beta$.

To this aim, we set
\begin{align*}
 \psi_i\bar e_i&:=e-\int|\bar v_i|^2-\Big[\psi_i^2\bar\delta_i +(1-\psi_i^2)\Big(e-\int|\bar v_{i-1}|^2\Big)\Big],\\
\varphi_i\tilde e_i&:=\Big(e-\int|\bar v_0|^2-\md_{i} e\Big)\chi_{[t_{i},t_{i-1}]}
\end{align*}
and rewrite \eqref{E:eineq_nsol} for $i=n$ as follows
\begin{align}
 e-\int|\bar v_n|^2&=\psi_n^2\bar\delta_n e+(1-\psi_n^2)\psi_{n-1}^2\bar\delta_{n-1} e+(1-\psi_n^2)(1-\psi_{n-1}^2)\Big(e-\int |\bar v_{n-2}|^2\Big)\notag\\
&+(1-\psi_n^2)\psi_{n-1}\bar e_{n-1}+\psi_n \bar e_n\notag\\
&\overset{\eqref{E_psivarphi2}}{=}\psi_n^2\bar \delta_n e+(1-\psi_n^2)\Big(e-\int |\bar v_{n-2}|^2\Big)+\psi_n\bar e_n\notag\\
&\overset{\eqref{E:vnsol}}{=}\psi_n^2\bar \delta_n e+(1-\psi_n^2)\Big(e-\int |\bar v_{0}|^2\Big)+\psi_n\bar e_n\notag\\
&=\chi_{[0,t_n]}\Big(e-|\bar v_0|^2\Big)+\chi_{\bigl[t_{n},1+h-\underset{i=1}{\overset{n}{\sum}}\ell_i\bigr]}\Big\{\bigl[(1-\psi_n^2)\md_n+\psi_n^2\bar\delta_n\bigr]e+(1-\psi_n^2)\varphi_n\tilde e_n+\psi_n \bar e_n\Big\}\label{E:inequsole}
\end{align}
Since \[
       |\tilde e_n|\leq\frac{\bar \beta}{2}\delta_n e, \quad |\bar e_n|\leq\frac{\bar \beta}{2}\bar\delta_n e
      \]
we find that 
\begin{align*}
 \Big|e-\int|\bar v_n|^2-\md_{n+1}e\Big|\chi_{[t_{n+1},t_n]}\leq \frac{\bar \beta}{2}\md_{n+1}e\\
 \Big|e-\int|\bar v_n|^2-\delta^\ast_{n}e\Big|\chi_{\bigl[t_{n},1+h-\underset{i=1}{\overset{n}{\sum}}\ell_i\bigr]}\leq {\bar \beta}\delta^\ast_{n}e
\end{align*}
which means that $\bar v_n$ satisfies the assumption \eqref{E_pert_ctheta_hp_ev} on $[t_n,1]$ with $\delta^\ast=\delta^\ast_n$, $\beta=2\bar\beta=\frac{1}{4}$ and, being $\bar v_n\chi_{[0,t_n]}=\bar v_0\chi_{[0,t_n]}$, it satisfies \eqref{E_pert_ctheta_hp_ev} also on $[t_{n+1}, t_n]$ with $\delta^\ast=\md_{n+1}$, $\bar\beta=\beta$.

As for the Reynolds stress tensor, set
\[
 \chi_{\{\psi_i>0\}}\tilde\mR_i:=\psi_i\bar\mR_{i,1}+\psi'\bar\mR_{i,2}
\]
and rewrite \eqref{E:Rineq_nsol} as 
\begin{align}
 \bar\mR_n&=(1-\psi_n^2)\bar \mR_{n-1}+\chi_{\{\psi_n>0\}}\tilde\mR_n\notag\\
&=(1-\psi_n^2)(1-\psi_{n-1}^2)\bar\mR_{n-2}+(1-\psi_{n}^2)\chi_{\{\psi_{n-1}>0\}}\tilde\mR_{n-1}+\chi_{\{\psi_n>0\}}\tilde \mR_n\notag\\
&\overset{\eqref{E_psivarphi2}}{=}(1-\psi_n^2)\bar\mR_0+\chi_{\{\psi_n>0\}}\tilde \mR_n\notag\\
&=\chi_{[0,t_n]}\bar\mR_0+\Big[(1-\psi_n^2)\bar\mR_0+\chi_{\{\psi_n>0\}}\bar\mR_n\Big]\chi_{[t_n,1]}.\notag
\end{align}
Therefore, 
\begin{align}
&\|\bar\mR_{n}-(1-\varphi_{n+1}^2)\mR'_{n+1}\|_0\chi_{[t_{n+1},t_n]}\leq\chi_{\{\varphi_{n+1}>0\}}\eta\bar\delta_n,\notag\\
&\|\mR'_{n+1}\|_0\leq\eta\bar\delta_{n-1}\notag\\
&\|\bar\mR_{n}-(1-\psi_{n}^2)\mR_{0}\|_0\chi_{\bigl[t_{n},1+h-\underset{i=1}{\overset{n}{\sum}}\ell_i\bigr]}\leq\chi_{\{\psi_{n}>0\}}\eta\bar\delta_n,\notag\\
&\|\chi_{[t_n,t_{n-1}]}\bar\mR_{0}\|_0\leq(1-\varphi_n^2)\eta\delta_{n-1}+\chi_{\{\varphi_n>0\}}\eta\delta_n\leq2\eta\md_n\label{E_lastrineq}
 \end{align}

Even though the last inequality \eqref{E_lastrineq} differs from the last in \ref{E_pert_ctheta_r0} by a factor $2$, choosing $\eta$ eventually smaller we can reduce to the assumptions of Proposition \ref{P_pert_ctheta}.
Hence, our claim is proved and we can apply Proposition \ref{P_pert_ctheta} on $[a_1,a_2]=[t_{n+1},1+h-\underset{i=1}{\overset{n}{\sum}}\ell_i]$ with $\psi=\psi_{n+1}$, $\bar\delta=\bar\delta_{n+1}=\delta_{n+2}$.

Thus we get \eqref{E:vnsol}-\eqref{E:Dineq_nsol} where $n$ is replaced by $n+1$ on the time interval $[t_{n+1}+\ell_{n+1},1+h-\underset{i=1}{\overset{n+1}{\sum}}\ell_i]$, $\ell_{n+1}<\min\{\tilde\ell_{n+1},\bar\ell_{n+1}\}$. 

Now observe that, since $\psi_{n+1}=0$ on $[0,t_{n+1}+\ell_{n+1}]$, we can extend trivially $v_{n+1}$ to be equal to $v_n$, $p_{n+1}$ to $p_n$ and $\mR_{n+1}$ to $\mR_n$ on $[0,t_{n+1}+\ell_{n+1}]$ and moreover also \eqref{E:vnsol}-\eqref{E:Dineq_nsol} extend to the larger time interval $[0,1+h-\underset{i=1}{\overset{n+1}{\sum}}\ell_i]$.

{\it \underline{Step 2:} Convergence to a solution}

By $\eqref{E_lnsol}$, the functions $\{(\bar v_n,\bar p_n,\bar\mR_n)\}$ are well defined on $\T\times[0,1]$.

From \eqref{E:vineq_nsol}, \eqref{E:Rineq_nsol} and \eqref{E:pineq_nsol} we immediately get that
\[
 \exists\,(\bar v, \bar p,\bar \mR)\in C^0(\T\times[0,1];\R^3\times\R\times\SS{0})\text{ s.t. }\underset{n\to+\infty}{\lim}\|\bar v_n-\bar v\|_{C^0}+\|\bar p_n-\bar p\|_{C^0}+\|\bar \mR_n-\bar\mR\|_{C^0}=0.
\]

In particular, there exists a constant $A=A(\eps,e,\bar C)$ such that
\[
 A_n\leq A, \quad\forall\,n\in\N,
\]
and from now onwards we will substitute it to $A_n$ into \eqref{E:Dineq_nsol}.

Moreover, $\bar \mR\equiv0$, $\bar v(0)=\bar v_0(0)$ and, by \eqref{E:inequsole}, $e=\int|\bar v|^2$ for all $t\in[0,1]$. 

{\it \underline{Step 3:} Convergence in the H\"older norm} To show that the sequence $\{\bar v_n(t)\}$ converges to $\bar v(t)$ in $C^\theta(\T;\R^3)$ with exponent $\theta<\frac{1}{16}$ we use the same argument of \cite{DS12H}, as in Proposition \ref{P_subsolution}. Let us find $c>0$ s.t., $\forall\,n\in\N$
\begin{equation}
\label{E:dnc2}
 \bar D_n\leq a^{cb^{n+1}}
\end{equation}
for some $a=a(\eps,e)$.
Notice that \eqref{E:dnc2} holds for $n=0$ and assume that it holds also $\forall\,i\leq n-1$. Then
\begin{align}
 \bar D_n&\overset{\eqref{E:Dineq_nsol}}{\leq} A\bar\delta_{n-1}^{\frac{3}{2}}\Big(\frac{\bar D_{n-1}}{\bar\delta_n^2}\Big)^{1+\eps}\bar \delta_n^{-\frac{17}{27}-\frac{4}{9}\eps}\notag\\
&=A\delta_n^{\frac{3}{2}}\Big(\frac{\bar D_{n-1}}{\delta_{n+1}^2}\Big)^{1+\eps}\delta_{n+1}^{-\frac{17}{27}-\frac{4}{9}\eps}\notag\\
&\leq A a^{-\frac{3}{2}b^n+b^{n+1}\Big(2(1+\eps)+\frac{17}{27}+\frac{4}{9}\eps\Big)+cb^{n}(1+\eps)}\notag\\
&=A a^{b^n\Big[c(1+\eps)+b\Big(\frac{44}{27}+\frac{22}{9}\eps\Big)\Big]}\notag
\end{align}
as in the proof of Proposition \ref{P_subsolution}, it is enough to check that $c=\frac{2(22+42\eps)}{9(1-2\eps)}$ satisfies
\begin{equation}
\label{E:cb2}
 cb+-c(1+\eps)-b\Big(\frac{44}{27}+\frac{22}{9}\eps\Big)\geq\eps.
\end{equation}
provided e.g. $\eps<\frac{1}{4}$.
If \eqref{E:cb2} holds, then
\[
 \bar D_n\leq A a^{cb^{n+1}} a^{-\frac{b^{n}}\eps}
\]
and \eqref{E:dnc2} follows provided $a\geq A^{\frac{1}{\eps}}$.

Finally estimate by interpolation
\begin{align}
 \|v_n-v_{n-1}\|_{C^\theta}&\leq\|v_n-v_{n-1}\|_{C^0}^{1-\theta}\|v_n-v_{n-1}\|_{C^1}^\theta\notag\\
&\leq a^{(-\frac{1}{2}(1-\theta)+cb\theta)b^{n-1}},\notag
\end{align}
which tends to zero as $n\to+\infty$ provided $(-\frac{1}{2}(1-\theta)+cb\theta)<0$. Hence, by the choice of $b$ and $c$ and letting $\eps$ tend to $0$, if $\theta<\frac{3}{47}$ and in particular $\theta<\frac{1}{16}$.

{\it \underline{Step 4:} Infinitely many solutions} As in Remark \eqref{R_inf_subsol}, one could try to adapt the argument of \cite{Cho} and show that the velocity field of a solution of \eqref{E:euler} constructed as in Steps $0$-$3$ can be made arbitrarily $H^{-1}$ close to the velocity field of the original subsolution. Since for any admissible initial datum there clearly exist infinitely many subsolutions, this would conclude the proof of Theorem \ref{T:main}. 

If the total kinetic energy is a constant, one can also argue in the following way. First notice that, given any admissible initial datum $v_0$ and any time $s\in (0,1)$, one can construct an admissible subolution $(v_s,p_s,
\mR_s)$ as in Proposition \ref{P_subsolution}, but supported in $[0,s)$. In the same way, given any other admissible initial datum $\bar v_0$ let $(\bar v_s,\bar p_s,
\bar \mR_s)$ be an admissible subsolution supported in $[0,s)$. Then, starting from the triple
\begin{equation*}
 (\tilde{v}_s,\tilde p_s,\tilde{\mR}_s)(x,t)=\left\{\begin{aligned}
  &(v_s,p_s,\mR_s)(x,t-ks) && &\text{if $t\in[ks,(k+1)s]$}\\
&(\bar v_s,\bar p_s,\bar \mR_s)(x,(k+2)s-t) && &\text{if $t\in[(k+1)s,(k+2)s]$}\\
 &(\bar v_s,\bar p_s,\bar \mR_s)(x,t-(k+2)s) && &\text{if $t\in[(k+2)s,(k+3)s]$}\\
&( v_s, p_s,\mR_s)(x,(k+4)s-t) && &\text{if $t\in[(k+3)s,(k+4)s]$}\
 \end{aligned}\right.
\end{equation*}
for every $k\in 4\mathbb N\cup\{0\}$, one can construct as in Steps 0-3 --but with cut-off functions that on $[k, (k+2)s]$ are symmetric w.r.t. $(k+1)s$ and on $[(k+2)s, (k+4)s]$ are symmetric w.r.t. $(k+3)s$--
a $4s$-periodic H\"older solution $(\hat v_s,\hat p_s)$ of \eqref{E:euler} which satisfies 
\begin{equation}
 \label{E:barhatv}
\hat v_s(ks)=v_s(0)=v_0,\qquad\hat v_s((k+2)s)=\bar v_s(0)=\bar v_0.
\end{equation}
 As $s$ varies in $(0,1)$, among these solutions there must be infinitely many different ones, because otherwise the only solution would be constant, thus contraddicting \eqref{E:barhatv}.

\section{Continuous solutions}
\label{S_cont}

\begin{definition}[Continuous subsolutions]
\label{D:contsubs}
Let $v_0\in C^0(\T;\R^3)$ such that $\int|v_0|^2=e(0)$. We say that a triple $(v,p,\mR)\in C^0(\T\times[0,1];\R^3\times\R\times\SS{0})\cap C^1(\T\times(0,1);\R^3\times\R\times\SS{0})$ is a \emph{continuous subsolution} of the Cauchy problem \eqref{E:euler} with initial datum $v_0$ if it solves \eqref{E:ersystem} on $\T\times(0,1)$ and $\exists\,\{t_n\}_{n\in\N}\subset(0,1)$ with $t_n$ decreasing to $0$ as $n\to+\infty$, $\{\delta_n\}_{n\in\N}\subset\Big(0,\frac{1}{2}\Big)$ with $\delta_n\leq\frac{1}{2}\delta_{n-1}$, $t_0=1$ such that
\begin{align}
&\frac{\Big(e(t)-\int|v(t)|^2\Big)(1-\delta_n)}{3(2\pi)^3}\Id-\mR(x,t)\in\mathcal M_3 \quad\forall\,t\in[t_{n},t_{n-1}],\,x\in\T\label{E:esubsdef}\\
&\underset{n=1}{\overset{\infty}{\sum}}\,\underset{[t_n,t_{n-1}]}{\max}\,\sqrt{e-\int_{\T}|v|^2}<+\infty,\label{E:dsubsdef}\\
&v(0)=v_0, \quad\mR(0)\equiv0.\notag
\end{align}
\end{definition}
Notice that, by \eqref{E:m3char}, \eqref{E:esubsdef} is equivalent to
\begin{equation*}
 \frac{\Big(e(t)-\int|v(t)|^2\Big)(1-\delta_n)}{6(2\pi)^3}\Id-\mR(x,t)\in\SS{+} \quad\forall\,t\in(0,1].
\end{equation*}

\begin{definition}[Admissible initial data]
We say that $v_0\in C^0(\T;\R^3)$ is an \emph{admissible initial datum} for the continuous Cauchy problem \eqref{E:euler} with prescribed kinetic energy $e$ if $\int |v_0|^2=e(0)$ and there exists a continuous subsolution as in Definition \ref{D:contsubs} with initial datum $v_0$.
\end{definition}

\begin{proposition}
\label{P_sol}
Given an admissible initial datum $v_0$ for the continuous Cauchy problem \eqref{E:euler}, there exist infinitely continuous solutions of \eqref{E:euler} with initial datum $v_0$. 
\end{proposition}

\begin{proposition}
\label{P_subsol}
 There exist infinitely many admissible initial data for the continuous Cauchy problem \eqref{E:euler}.
\end{proposition}

\begin{remark}
 In the proofs of Propositions \ref{P_sol} and \ref{P_subsol} we respectively show that there exist a solution and an admissible initial datum. The fact that they are infinitely many follows as in the Proofs of Theorem \ref{T:main} and Proposition \ref{P_subsolution}, therefore we omit the details. 
\end{remark}

The main step in the proof of both Propositions \ref{P_sol} and \ref{P_subsol} consists in being able to add to a subsolution a perturbation term in such a way that the system \eqref{E:ersystem} is still satisfied and both the energy gap and the supremum norm of the Reynolds stress are reduced, while keeping the $C^0$ norm of the new subsolution controlled by a sufficiently small parameter. The global character of our result is gained via multiplication by suitable time-dependent cut-off functions.

\begin{proposition}
 \label{P_pert_c0}
Let $e>0$, $e\in C^{\infty}([a_1,a_2])$. Then $\exists\, M>0$ depending only on $e$ such that the following holds. Let $(v,p,\mathring R)\in C^{1}(\T\times[a_1,a_2];\R^3\times\R\times \SS{0})$ solving \eqref{E:ersystem} and satisfying
\begin{align}
\label{E:reprcond}
 \frac{\Big(e-\int|v|^2\Big)(1-\delta)}{3(2\pi)^3}\Id-\mR\in\mathcal M_3,\qquad\forall\,t\in[a_1,a_2],x\in\T
\end{align}
for some $0<\delta\leq\frac{1}{2}$ and let $\psi\in C^1([a_1,a_2];[0,1])$, $\bar\delta\leq\frac{1}{2}\delta$. Then, there exists $(v_1,p_1,\mathring R_1)\in C^{1}(\T\times[a_1,a_2];\R^3\times\R\times \SS{0})$ which solves \eqref{E:ersystem} and satisfies, for all $t\in[a_1,a_2]$:
\begin{align}
&v_1=v+\psi w_1\label{E_v_decc}\\
&\mathring R_1=(1-\psi^2)\mathring R+\psi {\mathring R_{1,1}}+\psi'{\mathring R_{1,2}}\label{E_R_decc}\\
&p_1=p-\psi^2 p_o\label{E_p_decc}
\end{align}
together with the following estimates
\begin{align}
&\|v_1-v\|_0\leq \psi M\sqrt{\Delta}\label{E:w_1estc0}\\
 &\Big|e-\int_{\T}|v_1|^2-\Big[\psi^2\delta\Delta +(1-\psi^2)\Delta\Big]\Big|\leq\psi\bar\delta\underset{[a_1,a_2]}{\min}\Delta\label{E:enestpertc0}\\
&\|p_1-p\|_0\leq \psi^2{M^2}\Delta\label{E_pc0}\\
&\|\mathring R_1-(1-\psi^2)\mathring R\|_0\leq\chi_{\{\psi>0\}}\eta\bar\delta,\label{E:mrpertc0}
\end{align}
 where 
\[
  \Delta(t):=e(t)-\int_{\T}|v|^2(t)
\]
 and $\eta$ is given by
\begin{equation}
\label{E:etadef}
 \eta=\frac{r_0}{12(2\pi)^3}\underset{[a_1,a_2]}{\min}\Delta,
\end{equation}
with $r_0>0$ as in Lemma \ref{L:geom1}.
\end{proposition}

\begin{remark}
\label{R_iteration}
 Notice that, if $(v_1,p_1,\mR_1)$ is as in Proposition \ref{P_pert_c0}, then
\begin{equation}
\label{E:endelta2}
  \frac{\Big(e-\int|v_1|^2\Big)(1-\bar\delta)}{3(2\pi)^3}\Id-\mR_1\in\mathcal M_3,\quad\forall\,t\in[a_1,a_2],\,x\in\T.
\end{equation}
In order to prove \eqref{E:endelta2}, we first observe by \eqref{E:enestpertc0} that 
\begin{align*}
 \ev{1}(1-\bar\delta)=(1-\psi^2)\Delta(1-\bar\delta)+\psi^2\Delta(1-\bar\delta)\delta+\psi(1-\bar\delta)f_1,
\end{align*}
with $|f_1|\leq\bar\delta\underset{[a_1,a_2]}{\min}\Delta$.
Then, by \eqref{E_R_decc}
\begin{align}
 \frac{\ev{1}(1-\bar\delta)}{3(2\pi)^3}\Id-\mR_1&=(1-\psi^2)\Big[\frac{\Delta(1-\delta)}{3(2\pi)^3}\Id-\mR\Big]\label{endelta21}\\
&+\frac{(1-\psi^2)\Delta(\delta-\bar\delta)+\psi^2\Delta(1-\bar\delta)\delta+\psi f_1(1-\bar\delta)}{3(2\pi)^3}\Id\notag\\
&-\psi\mR_{1,1}-\psi'\mR_{1,2}.\notag
\end{align}
The summand in \eqref{endelta21} belongs to $\mathcal M_3$ by \eqref{E:reprcond}. As for the remaining terms, since $\delta\leq\frac{1}{2}$ and $\bar\delta\leq\frac{1}{2}\delta$,
\begin{align*}
 \Big|(1-\psi^2)\Delta(\delta-\bar\delta)+\psi^2\Delta(1-\bar\delta)\delta+\psi f_1(1-\bar\delta)\Big|&\geq\underset{[a_1,a_2]}{\min}\Delta\Big[\delta+\bar\delta\bigl(\psi^2(1-\delta)-\psi(1-\bar\delta)-1\bigr)\Big]\\
&\geq\underset{[a_1,a_2]}{\min}\Delta\Big[\delta-\bar\delta-\bar\delta\frac{(1-\bar\delta)^2}{4(1-\delta)}\Big]\\
&\geq\underset{[a_1,a_2]}{\min}\Delta\Big[\frac{\delta}{2}-\frac{\delta}{2}\frac{(1-\bar\delta)^2}{2}\Big]\\
&\geq\underset{[a_1,a_2]}{\min}\Delta\frac{\delta}{4}.
\end{align*}
Hence, by \eqref{E:mrpertc0} and \eqref{E:etadef}
\begin{align*}
\frac{\|\psi\mR_{1,1}+\psi'\mR_{1,2}\|_0 3(2\pi)^3}{\Big|(1-\psi^2)\Delta(\delta-\bar\delta)+\psi^2\Delta(1-\bar\delta)\delta+\psi f_1(1-\bar\delta)\Big|}\leq r_0.
\end{align*}

\end{remark}

\begin{proof}
 [Proof of Proposition \ref{P_pert_c0}.] $(v_1,p_1,\mR_1)$ is constructed as in the proof of Proposition \ref{P_pert_ctheta}, but replacing $v_\ell$, $\mR_\ell$ with $v$ and $\mR$, and $\rho_\ell$, $R_\ell$ respectively with
\begin{align*}
\rho(t)&= \frac{1}{3(2\pi)^3}\Delta(1-\delta),\\
R(x,t)&=\rho(t)\Id-\mathring R(x,t).
\end{align*}
More precisely, we define
\begin{align}
&w_o(x,t):=\psi\underset{j=1}{\overset{8}{\sum}}\underset{k\in\Lambda_j}{\sum}\sqrt{\rho(t)}\gamma^{(j)}_k\Big(\frac{R(x,t)}{\rho(t)}\Big)\phi^{(j)}_{k,\mu}(v(x,t),\lambda t)B_ke^{ik\cdot\lambda x},\notag\\
&w_c:=-\mathcal Q w_o\notag\\
&v_1:=\psi w_o+\psi w_c\notag\\
&p_1:=p-\psi^2 \frac{|w_o|^2}{2}\notag\\
 &\mathring R_{1}:=(1-\psi^2)\mR+\psi^2\mathcal R\Big[\div\Big( w_1\otimes w_1+\mR-\frac{|w_o|^2}{2}\Id\Big)\Big]+\psi\mathcal R\bigl[\partial_t w_1+\div(w_1\otimes v+v\otimes w_1)\bigr]+\psi'\mathcal R(w_1),\label{r5}
\end{align}
assuming w.l.o.g. that
\begin{equation*}
 \mu,\,\lambda,\,\frac{\lambda}{\mu}\in\N,\quad \mu\geq\Delta^{-1}.
\end{equation*}

In order to show that the main perturbation term $\psi w_o$ is well defined, it is then sufficient to show that $\frac{R}{\rho}\in\mathcal N$, where $\mathcal N\subset \mathcal M_3$ satisfies the assumptions of Lemma \ref{L:geom1}. To this aim it is sufficient to notice that
\begin{equation}
\label{E:equivm3}
 \frac{R}{\rho}\in\mathcal M_3\quad\Leftrightarrow\quad \Id-\frac{3(2\pi)^3}{\ev{}(1-\delta)}\mR\in\mathcal M_3,
\end{equation}
which is precisely \eqref{E:reprcond}.

Reasoning as in the proof of Proposition \ref{P_pert_ctheta}, we get the analogue of Proposition 6.1 of \cite{DS12} for the coefficients $a_k(s,y,\tau)=\sqrt{\rho(s)}\gamma^{(j)}_k\Big(\frac{R(y,s)}{\rho(s)}\Big)\phi^{(j)}_{k,\mu}(v(y,s),\tau)$, namely
\begin{align}
 \|\psi a_{k}(\cdot,s,\tau)\|_r&\leq\psi(s)C\sqrt{\Delta}\mu^r\label{E_62c}\\
 \|\partial_\tau \psi a_{k}(\cdot,s,\tau)\|_r&\leq\psi(s)C\sqrt{\Delta}\mu^r\label{E_63c}\\
\|(\partial_\tau \psi a_{k}+ i(k\cdot v)\psi a_{k})(\cdot,s,\tau)\|_r&\leq\psi(s)C\sqrt{\Delta}\mu^{r-1}\label{E_64c}\\
\|\partial_s \psi a_{k}(\cdot,s,\tau)\|_r&\leq\psi(s)C\sqrt{\Delta}\mu^{r+1}+|\psi'(s)|C\sqrt{\Delta}\mu^r.\label{E_65c}
\end{align}
We also have
\begin{equation}
\label{E_double_W}
       \psi W_o\otimes \psi W_o(y,s,\tau,\xi)= \psi^2(s)R(y,s)+\psi^2(s)\sum_{1\leq |k|\leq2\lambda_0}U_{k}(y,s,\tau)e^{ik\cdot\xi},
      \end{equation}
where $U_k\in C^\infty (\T\times[a_1,a_2]\times\R)$ satisfy 
\begin{equation}
 U_kk=\frac{1}{2}(\tr U_k)k
\end{equation}
 and the following estimates for any $r\geq0$:
\begin{align}
 \|\psi^2 U_{k}(\cdot,s,\tau)\|_r&\leq C\psi^2(s)\Delta\mu^r;\label{E_ukrc}\\
 \|\partial_\tau \psi^2 U_{k}(\cdot,s,\tau)\|_r&\leq C\psi^2(s)\Delta\mu^r;\\
 \|\partial_s\psi^2 U_{k}(\cdot,s,\tau)\|r&\leq C\psi^2(s)\Delta\mu^{r+1}+C|\psi'(s)|\Delta\mu^r;\\
 \|\partial_\tau \psi^2 U_{k}+i(k\cdot v) \psi^2 U_{k}(\cdot,s,\tau)\|_r&\leq C\psi^2(s)\Delta\mu^{r-1}.\label{E_ukrtc}
\end{align}

Notice that the above estimates are the same as in Proposition 6.1 of \cite{DS12} up to replacing $\delta$ with $\Delta$ everywhere and multiplying them by suitable powers of $\psi$ and its derivatives. Moreover, the constants $C$ may now depend, in contrast with the estimates for the H\"older continuous case, also on $C^r$ norms of $v$ and $\mR$. 

From \eqref{E_62c}-\eqref{E_ukrtc} we deduce as in Sections \ref{S_v_1_est} and \ref{S_energy_est} the following estimates:
\begin{align*}
\|\psi w_o\|_r\leq\sqrt{\Delta}\,\lambda^r,\qquad\forall\,r\geq0\\
\|\psi w_c\|_\alpha\leq\sqrt{\Delta}\,\mu\lambda^{\alpha-1},\qquad\forall\,\alpha>0\\
\Big|e-\int_{\T}|v_1|^2-\bigl[\psi^2\delta\Delta+(1-\psi^2)\Delta\bigr]\Big|\leq C\mu\lambda^{\alpha-1},
\end{align*}
 for any $\alpha\in\Big(0,\frac{\omega}{1+\omega}\Big)$.

Finally, setting
\begin{align*}
 \mR_{1,1}&:=\psi\mathcal R\bigl[\div\bigl( w_1\otimes w_1+\mR-\frac{|w_o|^2}{2}\Id\bigr)\bigr]+\mathcal R\bigl[\partial_t w_1+\div(w_1\otimes v+v\otimes w_1)\bigr],\\
\mR_{1,2}&:=\mathcal R(w_1)
\end{align*}
the Reynolds stress tensor $\mR_1=(1-\psi^2)\mR_0+\psi\mR_{1,1}+\psi'\mR_{1,2}$ can be estimated as in Section \ref{S_Reynolds_est} as follows
\begin{equation*}
 \|\mR_1-(1-\psi^2)\mR_0\|_0\leq \psi C\bigl[\sqrt{\Delta}\,\mu\lambda^{\alpha-1}+\sqrt{\Delta}\,\mu^{-1}\lambda^\alpha\bigr]+|\psi'|C\sqrt{\Delta}\,\mu\lambda^{\alpha-1}.
\end{equation*}
Since $\mu\leq\lambda$, we conclude that there exists a suitable choice of $\alpha$, $\mu$ and $\lambda$ for which \eqref{E:w_1estc0}-\eqref{E:mrpertc0} hold.
\end{proof}

\begin{proof} [Proof of Proposition \ref{P_subsol}.]
Let $(v_0,p_0,\mR_0)\equiv 0$. Observe that it trivially satisfies the assumptions of Proposition \ref{P_pert_c0} on  $[a_1,a_2]=[0,1]$ with $\delta=\frac{1}{2}$. Set $\delta_0=\frac{1}{2}$, $\Delta_0=\ev{0}=e$, $t_0=1$, $\psi_0\equiv1$, $\delta_n=\frac{1}{2^{n+1}}$, $\{t_n\}\subset(0,1)$ decreasing to $0$
and $\varphi_n\in C^1([0,1];[0,1])$ s.t.
\begin{equation*}
 \varphi_n=\left\{\begin{aligned}
               &1 && &[0,t_n]\\
               &0 && &[t_{n-1},1]
              \end{aligned}\right.
\end{equation*}
Notice that the functions $\psi=\varphi_n$, $\varphi=\varphi_{n-1}$ satisfy \eqref{12} for all $n\in\N$. In particular,
\begin{equation}
\label{E:varphikk'}
 \underset{i=k}{\overset{\bar k}{\prod}}(1-\varphi_{i}^2)=1-\varphi_{k}^2,\quad\underset{i=k}{\overset{\bar k}{\prod}}\varphi_i=\varphi_{\bar k},\quad (1-\varphi_i^2)\varphi_j=0 \text{ if $i<j$.}
\end{equation}

{\it \underline{Step 1.} Iterative perturbation step}  

Let us assume that, after $n$ steps of the iteration there exist $\{(v_i,p_i,\mR_i)\}_{i=1}^n$ solutions of \eqref{E:ersystem} on $[0,1]$ satisfying
\begin{align}
 &\Big|e-\int|v_i|^2-\Big[\varphi_i^2\Delta_{i-1}\delta_{i-1} +(1-\varphi_i^2)\Delta_{i-1}\Big]\Big|\leq\varphi_i\delta_{i}\underset{[0,t_{i-1}]}{\min}\Delta_{i-1}\label{E:eineq_nc0}\\
&\|v_i-v_{i-1}\|_0=\|\varphi_iw_i\|_0\leq\varphi_iM\sqrt{\Delta_{i-1}}\label{E:vineq_nc0}\\
&\|\mR_i-(1-\varphi_i^2)\mR_{i-1}\|_0\leq\chi_{\{\varphi_i>0\}}\eta_i\delta_{i}\label{E:Rineq_nc0}\\
&\|p_i-p_{i-1}\|_0\leq\varphi_i^2 M^2\Delta_{i-1},\label{E:pineq_nc0}
\end{align}
where 
\[
\Delta_i=e-\int_{\T}|v_i|^2,\qquad\eta_i=\frac{r_0}{12(2\pi)^3}\underset{[0,t_{i-1}]}{\min}\Delta_{i-1},
\]
and 
\begin{equation}
\label{E:evnn+1subs}
 \frac{\ev{i}(1-\delta_i)}{3(2\pi)^3}\Id-\mR_i\in\mathcal M_3,\quad\forall\,t\in[0,t_{i-1}].
\end{equation}
In particular,
\begin{equation*}
 v_i=v_{i-1},\quad p_i=p_{i-1},\quad \mR_i=\mR_{i-1}\quad\text{ on $[t_{i-1},1]$.}
\end{equation*}

Since $\supp\,\varphi_{n+1}\subset[0,t_n]$, we can apply Proposition \ref{P_pert_ctheta} to $(v_n,p_n,\mR_n)$ with $\bar\delta=\delta_{n+1}$, $\psi=\varphi_{n+1}$ and get $(v_{n+1},p_{n+1},\mR_{n+1})$ solutions of \eqref{E:ersystem} on $[0,1]$ satisfying \eqref{E:eineq_nc0}-\eqref{E:evnn+1subs} with $i=n+1$.

{\it \underline{Step 2.} Convergence to a subsolution} Being
\begin{equation}
\label{E:vmvn}
 v_m\chi_{[t_{n},t_{n-1}]}=v_{n}\chi_{[t_{n},t_{n-1}]},\quad p_m\chi_{[t_{n},t_{n-1}]}=p_{n}\chi_{[t_{n},t_{n-1}]},\quad\mR_m\chi_{[t_{n},t_{n-1}]}=\mR_{n}\chi_{[t_{n},t_{n-1}]}\quad\forall\,m\geq n\in\N,
\end{equation}
$(v_n,p_n, \mR_n)\rightarrow(v,p,\mR)$ in $C^0(\T\times(0,1];\R^3\times\R\times\SS{0})$, where 
$(v,p,\mR)$ solves the Euler-Reynolds system \eqref{E:ersystem} and satisfies \eqref{E:esubsdef} on $\T\times(0,1)$.

To show full convergence on $[0,1]$ to a subsolution, we have to prove
\begin{equation}
\label{E:deltaest}
 \underset{n=1}{\overset{+\infty}{\sum}}\,\underset{[0,t_{n-1}]}{\max}\,\sqrt{\Delta_{n-1}}<+\infty.
\end{equation}
Indeed, if \eqref{E:deltaest} holds, not only $(v_n,p_n,\mR_n)\to (v,p,\mR)$ in $C^0(\T\times[0,1])$, but also $\mR(0)\equiv 0$ and $e(0)=\int_{\T}|v(0)|^2$. Moreover, by \eqref{E:vmvn} amd \eqref{E:eineq_nc0}, $v$ also satisfies \eqref{E:dsubsdef}.

Set
\[
 \varphi_i f_i:=e-\int|v_i|^2-\bigl[\varphi_i^2\Delta_{i-1}\delta_{i-1} +(1-\varphi_i^2)\Delta_{i-1}\bigr].
\]
Expanding \eqref{E:eineq_nc0} for $i=n$ we obtain
\begin{align*}
 \Delta_n&=\varphi_n^2\delta_{n-1}\Delta_{n-1}+(1-\varphi_n^2)\Delta_{n-1}+\varphi_n f_n\\
&=\varphi_n^2\varphi_{n-1}^2\delta_{n-1}\delta_{n-2}\Delta_{n-2}+\varphi_n^2(1-\varphi_{n-1}^2)\delta_{n-1}\Delta_{n-2}+(1-\varphi_n^2)\varphi_{n-1}^2\delta_{n-2}\Delta_{n-2}\\
&+(1-\varphi_n^2)(1-\varphi_{n-1}^2)\Delta_{n-2}+\varphi_n^2\varphi_{n-1}\delta_{n-1}f_{n-1}+(1-\varphi_n^2)\varphi_{n-1}f_{n-1}+\varphi_nf_n\\
&\overset{\eqref{E:varphikk'}}{=}\varphi_n^2\delta_{n-1}\delta_{n-2}\Delta_{n-2}+(1-\varphi_n^2)\varphi_{n-1}^2\delta_{n-2}\Delta_{n-2}+(1-\varphi_{n-1}^2)\Delta_{n-2}+\varphi_n^2\delta_{n-1} f_{n-1}\\
&+(1-\varphi_n^2)\varphi_{n-1}f_{n-1}+\varphi_nf_n.
\end{align*}
Hence
\begin{align*}
 \Delta_{n}\chi_{[0,t_n]}&=\delta_{n-1}\delta_{n-2}\Delta_{n-2}+\delta_{n-1} f_{n-1}+f_n=\Big(\underset{j=0}{\overset{n-1}{\prod}}\delta_j\Big)\Delta_0+\underset{i=1}{\overset{n-1}{\sum}}\Big(\underset{j=i}{\overset{n-1}{\prod}}\delta_j\Big)f_i+f_n
\end{align*}
which, by the choice of the $\delta_n$ and \eqref{E:eineq_nc0}, implies \eqref{E:deltaest}.
 
\end{proof}

\begin{proof}
 [Proof of Proposition \ref{P_sol}]
Let $(\bar v_0,\bar p_0,\bar \mR_0)$ be a subsolution and set $\bar\Delta_0=e-\int|\bar v_0|^2$. Let $\bar\delta_n=\delta_{n+1}$ for all $n\in\N\cup\{0\}$. 
Let $t_0=1$, $\{t_n\}$, $\bar\delta_0$, $\{\bar\delta_n\}$ as in Definition \ref{D:contsubs} and 
$\psi_0\equiv 0$, $\psi_n\in C^1([0,1])$ s.t.
\begin{equation*}
 \psi_n=\left\{\begin{aligned}
         &1 && &[t_{n-1},1]\\
         &0 && &[0,t_n]
        \end{aligned}\right.
\end{equation*}
Notice that the functions $\psi=\psi_{n-1}$, $\varphi=\psi_{n}$ satisfy \eqref{12}  for all $n\in\N$.
As a consequence,
\begin{equation}
\label{E:phipsi2c0}
 \underset{i=k}{\overset{\bar k}{\prod}}(1-\psi_{i}^2)=1-\psi_{\bar k}^2,\quad\underset{i=k}{\overset{\bar k}{\prod}}\psi_i=\psi_{k},\quad (1-\psi_i^2)\psi_j=0 \text{ if $i>j$.}
\end{equation}
Observe that $(\bar v_0,\bar p_0, \bar \mR_0)$ satisfies the assumptions of Proposition \ref{P_pert_c0} on $[a_1,a_2]=[t_1,t_0]$ with $\delta=\bar\delta_0$.

{\it \underline{Step 2.} Iterative perturbation step}  

Let us assume that, after $n$ steps, there exist $\{(\bar v_i,\bar p_i,\bar{\mR}_i)\}_{i=1}^n$ solutions of \eqref{E:ersystem} on $[0,1]$ satisfying
\begin{align}
 &\Big|e-\int|\bar v_i|^2-\Big[\psi_i^2\bar\delta_{i-1}\bar\Delta_{i-1} +(1-\psi_i^2)\bar\Delta_{i-1}\Big]\Big|\leq\psi_i\bar\delta_i\underset{[t_{i},1]}{\min}\,\bar\Delta_{i-1}\label{E:eineq_nsolc0}\\
&\|\bar v_i-\bar v_{i-1}\|_0\leq\psi_iM\sqrt{\bar\Delta_{i-1}}\label{E:vineq_nsolc0}\\
&\|\bar\mR_i-(1-\psi_i^2)\bar\mR_{i-1}\|_0\leq\chi_{\{\psi_i>0\}}\eta_i\bar\delta_i\label{E:Rineq_nsolc0}\\
&\|\bar p_i-\bar p_{i-1}\|_0\leq\psi_i^2 M^2\bar\Delta_{i-1}\label{E:pineq_nsolc0}\\
\end{align}
with 
\[
\bar\Delta_i=e-\int|\bar v_i|^2,\qquad\eta_i=\frac{r_0}{12(2\pi)^3}\underset{[t_{i},1]}{\min}\bar\Delta_{i-1}.  
\]
and
\begin{equation}
\label{E:evnn+1sol}
 \frac{\Big(e-\int|\bar v_i|^2\Big)(1-\bar\delta_i)}{3(2\pi)^3}\Id-\bar\mR_i\in\mathcal M_3,\quad\forall\,t\in[t_{i+1},1].
\end{equation}
Observe that 
\begin{equation}
\label{vii-1}
 v_i=v_{i-1},\quad p_i=p_{i-1},\quad \mR_i=\mR_{i-1}\quad\text{ on $[0,t_{i}]$.}
\end{equation}

Thus we can apply Proposition \ref{P_pert_c0} to $(\bar v_n, \bar p_n,\bar \mR_n)$ on $[a_1,a_2]=[t_{n+1},1]$ with $\bar\delta=\bar\delta_{n+1}$, $\psi=\psi_{n+1}$ and get $(\bar v_{n+1}, \bar p_{n+1},\bar\mR_{n+1})$ satisfying \eqref{E:eineq_nsolc0}-\eqref{E:pineq_nsolc0} with $i=n+1$.
By Remark \ref{R_iteration} $(\bar v_{n+1}, \bar p_{n+1},\bar\mR_{n+1})$ satisfies \eqref{E:evnn+1sol} on $[t_{n+1},1]$. Since, by \eqref{vii-1}, $(\bar v_{n+1}, \bar p_{n+1},\bar\mR_{n+1})\equiv(\bar v_{0}, \bar p_{0},\bar\mR_{0})$ on $[0,t_{n+1}]$, then, by \eqref{E:esubsdef}, \eqref{E:evnn+1sol} holds also on $[t_{n+2}, t_{n+1}]$.

{\it \underline{Step 2.} Convergence to a solution}
As in the proof of Proposition \ref{P_subsol}, we have to show that
\begin{equation}
\label{E:deltaestsol}
 \underset{n=1}{\overset{+\infty}{\sum}}\,\underset{[t_{n},1]}{\max}\,\sqrt{\bar \Delta_{n-1}}<+\infty.
\end{equation}
  Set \[
        \psi_i\bar f_i:=\bar \Delta_i-\bigl[\psi_i^2\bar\delta_{i-1}\bar\Delta_{i-1} +(1-\psi_i^2)\bar\Delta_{i-1}\bigr].
       \]
Expanding \eqref{E:eineq_nsolc0} we get
\begin{align*}
 \bar\Delta_n&=\psi_n^2\bar\delta_{n-1}\bar\Delta_{n-1}+(1-\psi_n^2)\bar\Delta_{n-1}+\psi_n \bar f_n\\
&=\psi_n^2\psi_{n-1}^2\bar\delta_{n-1}\bar\delta_{n-2}\bar\Delta_{n-2}+\psi_n^2(1-\psi_{n-1}^2)\bar\delta_{n-1}\bar\Delta_{n-2}+(1-\psi_n^2)\psi_{n-1}^2\bar\delta_{n-2}\bar\Delta_{n-2}\\
&+(1-\psi_n^2)(1-\psi_{n-1}^2)\bar\Delta_{n-2}+\psi_n^2\psi_{n-1}\bar\delta_{n-1}\bar f_{n-1}+(1-\psi_n^2)\psi_{n-1}\bar f_{n-1}+\psi_n\bar f_n\\
&\overset{\eqref{E:phipsi2c0}}{=}\psi_{n-1}^2\bar\delta_{n-1}\bar\delta_{n-2}\bar\Delta_{n-2}+\psi_n^2(1-\psi_{n-1}^2)\bar\delta_{n-1}\bar\Delta_{n-2}
+(1-\psi_n^2)\bar\Delta_{n-2}+\psi_{n-1}\bar\delta_{n-1}\bar f_{n-1}+\psi_n\bar f_n\\
&=(1-\psi_n^2)\bar\Delta_0+\underset{i=1}{\overset{n}{\sum}}\psi_i^2(1-\psi_{i-1}^2)\Big(\underset{j=i-1}{\overset{n-1}{\prod}}\bar\delta_j\Big)+\underset{i=1}{\overset{n-1}{\sum}}\psi_i\Big(\underset{j=i}{\overset{n-1}{\prod}}\bar\delta_j\Big)f_i+\psi_n f_n.
\end{align*}
By \eqref{E:dsubsdef} on $[t_{n+1},t_n]$, and by the choice of $\bar\delta_j$ and \eqref{E:eineq_nsolc0} we finally obtain \eqref{E:deltaestsol}.

Hence, $(\bar v_n, \bar p_n, \bar \mR_n)\to(\bar v,\bar p, \bar \mR)$ in $C^0(\T\times[0,1])$, with $\bar\mR\equiv 0$, $\bar v(0)\equiv \bar v_n(0)=\bar v_0(0)$, total kinetic energy $\int|\bar v|^2=e$ and solving \eqref{E:euler}.

\end{proof}

%\bibliography{Eulerbib}{}
%\bibliographystyle{plain}

\end{document}